\documentclass[a4paper, english, 10pt]{smfart} 

\usepackage{placeins} 

\usepackage{setspace}
\usepackage[british]{babel}

\usepackage[OT2,T1]{fontenc}

\usepackage{amssymb} 
\usepackage{mathrsfs} 

\usepackage{stmaryrd}
\usepackage{amsthm}

\usepackage{amsmath}
\usepackage[latin1]{inputenc}

\allowdisplaybreaks[1]

\usepackage{graphicx}
\usepackage{manfnt} 

 \usepackage{booktabs} 

\usepackage[margin=3cm]{geometry}

\usepackage[all]{xy}

\usepackage{math tools} 

\usepackage{enumitem}

\usepackage[pagebackref=true, colorlinks=true, linkcolor=black, citecolor=black, urlcolor=black]{hyperref}
\usepackage{smfthm}

\usepackage[msc-links, alphabetic]{amsrefs}

\DeclareSymbolFont{cyrletters}{OT2}{wncyr}{m}{n}
\DeclareMathSymbol{\Sha}{\mathalpha}{cyrletters}{"58}

\newtheorem{theoA}{Theorem}

\newtheorem*{coro*}{Corollary}
\newtheorem*{conj*}{Conjecture}
\newtheorem*{lemm*}{Lemma}

\providecommand{\twomat}[4]{\left(\begin{array}{cc}#1&#2\\#3&#4\end{array}\right)}
\providecommand{\smalltwomat}[4]{\left(\begin{smallmatrix}#1&#2\\#3&#4\end{smallmatrix}\right)}

\theoremstyle{definition} 

\newtheorem*{defi*}{Definition}

\theoremstyle{remark}
\newtheorem{remark*}{Remark}

\NumberTheoremsIn{subsection}

\numberwithin{equation}{subsection}
\numberwithin{table}{subsection}

\newcommand{\ot}{\otimes}

\newcommand{\ts}{\times}
\newcommand{\cd}{\cdot}

\newcommand{\beq}{\begin{equation}\begin{aligned}}
\newcommand{\eeq}{\end{aligned}\end{equation}}

\newcommand{\beqq}{\begin{equation*}\begin{aligned}}
\newcommand{\eeqq}{\end{aligned}\end{equation*}}

\newcommand{\lb}[1]{\label{#1}}

\newcommand{\divisor}{\mathrm{div}}

\newcommand{\Q}{\mathbf{Q}}

\newcommand{\GL}{\mathrm{GL}}

\newcommand{\X}{\mathscr{X}}

\newcommand{\cC}{\mathscr{C}}

\newcommand{\onto}{\twoheadrightarrow}

\newcommand{\cF}{\mathscr{F}}

\newcommand{\R}{\mathbf{R}}
\newcommand{\Z}{\mathbf{Z}}

\newcommand{\frakm}{\mathfrak{m}}
\newcommand{\frakp}{\mathfrak{p}}

\newcommand{\Ig}{\mathrm{Ig}}
\newcommand{\an}{\mathrm{an}}

\newcommand{\into}{\hookrightarrow}

\newcommand{\Y}{\mathscr{Y}}

\newcommand{\nrd}{\mathrm{nrd}}

\newcommand{\tht}{\theta}
\newcommand{\lm}{\lambda}
\newcommand{\Lm}{\Lambda}
\newcommand{\sg}{\sigma}
\newcommand{\Sg}{\Sigma}

\newcommand{\cI}{\mathscr{I}}
\newcommand{{\cG}}{\mathscr{G}}

\newcommand{\cM}{\mathscr{M}}
\newcommand{\cN}{\mathscr{N}}

\newcommand{\cB}{B}

\newcommand{\cD}{\mathscr{D}}

\newcommand{\C}{\mathbf{C}}

\newcommand{\N}{\mathbf{N}}
\newcommand{\OO}{\mathscr{O}}
\newcommand{\A}{\mathbf{A}}

\newcommand{\bks}{\backslash}
\newcommand{\baar}{\overline}

\newcommand{\eps}{\varepsilon}
\newcommand{\vphi}{\varphi}
\newcommand{\vpi}{\varpi}

\newcommand{\vol}{\mathrm{vol}}

\newcommand{\Gal}{\mathrm{Gal}}

\newcommand{\Ker}{\mathrm{Ker}\,}

\newcommand{\Hom}{\mathrm{Hom}\,}
\newcommand{\End}{\mathrm{End}\,}

\newcommand{\llb}{\llbracket}
\newcommand{\rrb}{\rrbracket}

\newcommand{\Spec}{\mathrm{Spec}\,}
\newcommand{\Spf}{\mathrm{Spf}\,}

\newsavebox\tempbox
\let\svwidetilde\widetilde
\renewcommand\widetilde[1]{\sbox\tempbox{$#1$}\svwidetilde{\usebox{\tempbox}}}

   \def\XXint#1#2#3{{\setbox0=\hbox{$#1{#2#3}{\int}$}
        \vcenter{\hbox{$#2#3$}}\kern-.5\wd0}}

\title{$p$-adic equidistribution of CM points }
\author{Daniel Disegni} 
\address{Department of Mathematics, Ben-Gurion University of the Negev, Be'er Sheva 84105, Israel.}
\email{disegni@bgu.ac.il}
\begin{document}
\begin{abstract}
Let $X$ be a modular curve and consider  a sequence of Galois orbits of CM points in $X$, whose $p$-conductors tend to infinity.  Its  equidistribution properties  in $X(\C)$ and in the reductions of $X$ modulo primes different from $p$ are well understood.
 We study the equidistribution problem  in the Berkovich analytification $X_{p}^{\an}$ of $X_{\Q_{p}}$. 
 
 We partition the set of CM points of sufficiently high conductor in $X_{\Q_{p}}$ into finitely many explicit
\emph{basins} $\cB_{V}$, indexed by the irreducible components $V $ of the mod-$p$ reduction of the canonical model of $X$. We prove   that  a sequence $z_{n}$ of local  Galois orbits of CM points with $p$-conductor going to infinity has a limit in $X_{p}^{\an}$ if and only if it is eventually supported in a single basin $\cB_{V}$. If so, the limit is the unique point of $X_{p}^{\an}$ whose mod-$p$ reduction is the generic point of $V$. 

The result is proved in  the more general setting of  Shimura curves over totally real fields. The proof combines Gross's theory of quasicanonical liftings with a new formula for the  intersection numbers of CM curves and vertical components in a Lubin--Tate space.
\end{abstract}
\thanks{This work was supported by BSF grant 2018250 and ISF grant 1963/20.}
\maketitle

\tableofcontents

 \section{Introduction and statements of the main results}
 
The equidistribution properties of CM points have been studied from various points of view, with remarkable applications ranging from cases of the Andr\'e--Oort conjecture to the non-triviality of Heegner points. The main results until recently have been archimedean or $\ell$-adic.

 To illustrate the situation,  let $X_{/\Q}$ be a geometrically connected modular curve, and let  $$(z_{n})_{n\in \N}$$ be  a sequence of  CM points on $X$
 with $p$-part of the conductor going to infinity. We view the $z_{n}$ as scheme theoretic points of $X$, or equivalently as Galois orbits of geometric points.

Duke \cite{duke}, Clozel--Ullmo \cite{CU}, Zhang \cite{zhang} and others considered the images of the $z_{n}$ in $X(\C)$, and proved 
 the following archimedean equidistribution result: the sequence $\mu_{n}$ of averaged delta measures at the orbits $z_{n}$ converges to the hyperbolic probability measure on $X(\C)$. Cornut and Vatsal \cite{CV}  considered the reduction modulo $\ell$ of the $z_{n}$, when $\ell \neq p$ is a prime nonsplit in the fields of complex multiplication (and more generally, simultaneous reductions at  a finite set of such primes); they proved that 
the images of the orbits $z_{n}$ equidistribute to the counting probability measure on the finite set $\X_{\baar{\bf F}_{\ell}}^{\rm ss}$ of supersingular points. Very recently, Herrero--Menares--Rivera-Letelier \cite{HMRL-k} described the set of accumulation measures of the $\mu_{n}$ on the Berkovich analytic curve attached to the base-change $X_{\Q_{\ell}}$.

\medskip 

The purpose of this paper is to prove a $p$-adic equidistribution result for sequences of CM points.  
Consider the Berkovich analytic curve $X_{p}^{\an }$ attached to $X_{{\Q}_{p}}$.  It is a compact Hausdorff topological space, containing the set of closed points of $X_{\Q_{p}}$ and  equipped with a  reduction map to  the special fibre  $|\X_{{\bf F}_{p}}|$ of any model $\X/{\Z}_{p}$ of $X_{{\Q}_{p}}$. The generic point of any irreducible component $V$ of $|\X_{{\bf F}_{p}}|$ is the reduction of a unique point $\zeta_{V}\in X_{p}^{\an}$, called the \emph{Shilov point} of $V$. 
The study of our problem leads naturally to consider the canonical model defined by  Katz--Mazur \cite{KM}, whose  special fibre is topologically a union of finitely many irreducible curves intersecting at the supersingular points. 

Our main result is best described  in terms of sequences $(z_{n})$ of CM points of $X_{\Q_{p}}$, equivalently local Galois orbits of geometric CM points.  We explicitly partition the set of such CM points of  sufficiently large $p$-adic conductor  into  finitely many \emph{basins} $\cB_{V}$ indexed by the irreducible components of the special fibre of $\X_{{\bf F}_{p}}$. Our main theorem is that the sequence of measures $\mu_{n}:= \delta_{z_{n}}$ on $X_{p}^{\rm an}$  has a limit if and only if the sequence $(z_{n})$ is eventually supported in a single basin $\cB_{V}$, in which case the $\mu_{n}$ converge to the delta measure at $\zeta_{V}$.  Equivalently, the sequence $(z_{n})$ converges to $\zeta_{V}$ in the plain topological sense -- so that  we will only use the language of topology in what follows.

The theorem is proved in the more general context of Shimura curves over totally real fields. The rest of this introduction is dedicated to explaining its statement and the idea of proof, as well as the intersection formula that  lies at its core.  

For the modular curve of level~$1$, which has a single Shilov point, the equidistribution result was previously established in \cite{herrero}, as briefly discussed after Theorem \ref{main thm} below.

\subsection{CM points on Shimura curves and their integral models} \lb{sec 11}
Let $F$ be  a totally real field. Let $D$ be a quaternion algebra over $F$ whose ramification set $\Sg$ contains all the infinite places but one, which we denote by $\sg\colon F\into \R$.  
We may attach to $D$ a tower of Shimura curves $X_{U}/F$, where $U$ runs over compact open subgroups of $D_{\A^{\infty}}^{\ts}:=(D\ot_{\Q}\A^{\infty})^{\ts}$ (here $\A^{\infty}$ denotes the finite ad\`eles of $\Q$, whereas $\A$ will denote the ad\`eles).  This curve and its canonical integral model were studied by Carayol  \cite{carayol}, to which we refer for more details (see also \cite[\S 5]{asian} for a discussion of CM points).

 The points of $X_{U}^{}$ over $\C\stackrel{\sg}{\hookleftarrow} F$ can be  identified with 
\beq
\lb{Cpts}
X_{U}^{}(\C)\cong D^{\ts} \bks \mathfrak{h}^{\pm}\ts D_{\A^{\infty}}^{\ts}/ U \cup\{\rm cusps\}
\eeq
where $ \mathfrak{h}^{\pm} \cong \C-\R$ is viewed as a quotient of $D_{\sg}^{\ts}\cong \GL_{2}(\R)$, and  the finite set `$\{\rm cusps\}$' is non-empty only if $D=M_{2}(\Q)$ (it plays no role in this work).

From now on we fix an arbitrary  level $U\subset D_{\A^{\infty}}^{\ts}$,  assumed 
to be  sufficiently small so that   the $D^{\ts}$-action in \eqref{Cpts} has no nontrivial stabilisers. We let $X:=X_{U}$. 

\paragraph{CM points} Let $E$ be a CM quadratic extension of $F$, such  that each finite  $v\in\Sg$ is nonsplit in $E$.
Then the set of  $F_{\A}$-algebra embeddings $\psi\colon E_{\A}\into D_{\A}$ is non-empty.
  For any such $\psi$, the group $\psi(E^{\ts})\subset D_{\A^{\infty}}^{\ts}$ acts on the right on $X$.
 The fixed-point subscheme $X^{\psi(E^{\ts})} $,  called the scheme of points with CM by $(E, \psi)$,     is isomorphic (as $F$-subscheme)
to $\Spec E_{[\psi^{-1}(U)]}$ where $E_{C}$ denotes the abelian extension of $E$ with Galois group $C\subset E^{\ts}\bks E_{\A^{\infty}}^{\ts}$.  (This follows from the theory of complex multiplication as described in \cite[\S~5.2]{asian}.)
The \emph{CM (ind)-subscheme} of $X_{U}$ is the union 
$$X^{\rm CM}  =\bigcup_{(E, \psi)}  X^{\psi(E^{\ts})}.$$

\paragraph{Local integral models and their irreducible components} Fix a finite  place $v $ of $F$
  not in $\Sg$ and   an identification $D_{v}=M_{2}(F_{v})$, and assume $U=U^{v}U_{v}$ with $U_{v}\subset D_{v}^{\ts}$. Let  $\OO_{F_{v}}\subset F_{v}$ be the ring of integers, $\vpi_{v}$  a uniformiser, ${\bf F}_{v}$ the residue field.

Carayol \cite{carayol}, generalising \cite{KM}, defined a canonical integral model $\X=\X_{U}$ of $X_{}$ over  $\OO_{F_{v}}$. It carries a sheaf  $\cG$ of divisible  $\OO_{F_{v}}$-modules of rank $2$ together with, if $U_{v}=U_{n,v}:=\Ker(\GL_{2}(\OO_{F_{v}})\to \GL_{2}(\OO_{F_{v}}/\vpi_{v}^{n}))$, a 
Drinfeld level structure\footnote{This notion is recalled in \S\ref{LTs} below. When $n=0$, there is no level structure and the integral model was already constructed in \cite{morita}.}
$$\alpha\colon (\vpi_{v}^{-n}\OO_{F_{v}}/\OO_{F_{v}})^{2} \to \cG(\X).$$
In general, $\X=\X_{U}:= \X_{U^{v}U_{n,v}} /(U_{v}/U_{n,v})$ for any $U_{n,v}\subset U_{v}$.

The special fibre $\X_{{\bf F}_{v}}$ (see \cite[\S 9.4]{carayol}) is a union of connected components permuted simply transitively by an action of $F^{\ts}\bks F_{\A^{\infty}}^{\ts}/\OO_{F_{v}}^{\ts}\det U^{v}$.  The (\emph{supersingular}) locus  $\X_{{\bf F}_{v}}^{\rm ss}$
 where $\cG$ is connected is finite; its complement  $\X_{{\bf F}_{v}}^{\rm ord}:= \X_{{\bf F}_{v}} - \X_{{\bf F}_{v}}^{\rm ss}$ (\emph{ordinary} locus) is smooth.
 Each connected component $C\subset \X_{{\bf F}_{v}}$ is topologically a union of  irreducible components, each intersecting all of the others in each of the finitely many supersingular points.
 
 The irreducible components $V$ within a given connected component $C$ are canonically parametrised by\footnote{In \cite[\S 9.4]{carayol}, the discussion is over the algebraic closure $\baar{\bf F}_{v}$ (and the parameter is dual to our $\lm$). However, as is clear from the defining condition,  the irreducible components within each (${\bf F}_{v}$-)connected component are already defined over ${\bf F}_{v}$.}
$$ U_{v}\bks {\bf P}^{1}(F_{v}) ,$$
where we view $ {\bf P}^{1}(F_{v})$ as the space of $\OO_{F_{v}}^{\ts}$-classes of injections $\lm\colon \OO_{F_{v}}\to \OO_{F_{v}}^{2}$ with saturated image; for any $\OO_{F_{v}}$-module $M$, we will still denote by $\lm \colon M\to M^{2}$ the induced injection. The  parametrisation is  simple to  describe. Let 
\beq\lb{pontry}( -)^{*}:=\Hom_{\OO_{F_{v}}}(-, F_{v}/\OO_{F_{v}})= \varinjlim_{n} \ ( -)_{n}^{*},
\qquad   ( -)_{n}^{*}:=\Hom_{\OO_{F_{v}}}(-, \vpi_{v}^{-n}\OO_{F_{v}}/\OO_{F_{v}}), \eeq
be the Pontryagin duality functors. Throughout this paper, we  identify an irreducible component of a scheme (which by definition is an irreducible component of the underlying topological space)  with the corresponding subscheme with the reduced structure. Then the component $V(\lm)_{U}^{(C)}\subset C$
is the locus where the Drinfeld level structure $\alpha$ 
   factors through any quotient  
 $$\lm_{n}^{*}\colon (\vpi_{v}^{-n}\OO_{F_{v}}/\OO_{F_{v}})^{2} \to \vpi_{v}^{-n}\OO_{F_{v}}/\OO_{F_{v}}$$  in the  $ \OO_{F_{v}}^{\ts} \ts U_{v}/U_{n,v}$-class determined by $\lm_{n}^{*}$. 

\paragraph{CM points, their reductions, and basins}
We consider the base-change (ind-)scheme $X_{F_{v}}^{\rm CM}\subset X_{F_{v}}$ and refer to its points as the CM points in $X_{F_{v}}$.  It is well known that if $z\in X_{F_{v}}^{\rm CM} $ is a  point with CM by $(E, \psi)$, its reduction is a supersingular point if and only if $v$ is nonsplit in $E$.  Assume that this is the case. The ring  $\OO_{z}:={\rm End}_{\OO_{F}\text{-Mod}}(\cG_{z})$ is  the $\OO_{F_{v}}$-order in $E_{v}$ with unit group  $ \psi_{v}^{-1}(U_{0, v})$. Such orders are classified: there is a unique integer $s\geq 0$, called the ($v$-)\emph{conductor} of $z$, such that 
$$\OO_{z} =  \OO_{E_{v}}[s] := \OO_{F_{v}}+\vpi_{v}^{s}\OO_{E_{v}}.$$
By \cite[\S~5]{gross}, if $z$ has $v$-conductor $s$ there  is an $ \OO_{E_{v}, s}^{*}$-linear isomorphism, unique up to $\OO_{E_{v}}[s]^{\ts} $,
\beq\lb{beta}\beta \colon \OO_{E_{v}}[s]^{*} \to \cG_{z}(\baar{F}_{v}), \eeq
where  $( -)^{*} = \eqref{pontry}$.

\subsection{Partition into basins and the main result}
Let $z\in X_{F_{v}}^{\rm CM}$ be a point with CM by $E$ of $v$-conductor $s$,  keep the assumption that $E_{v}$ is nonsplit, and let $\beta$ be an $E$-level structure   on $\cG_{z}$ as in \eqref{beta}. 
Let $\alpha\colon (F_{v}/\OO_{v})^{2} \to \cG_{z}(\baar{F_{v}})$ be any isomorphism extending the level structure $\alpha_{z}$ on $\cG_{z}$, let $\beta$ be as in \eqref{beta}, and let 
\beq \lb{tau*}
\tau^{*}:=\beta^{-1}\circ \alpha\colon (F_{v}/\OO_{F_{v}})^{2}\to  
\OO_{E_{v}}[s]^{*}.\eeq
Let $\tau\colon E\to F^{2}$ be the $F$-linear extension of the Pontryagin-dual map to $\tau^{*}$. The class
$$ L(\tau):= U_{v}. L_{s}(\tau), \qquad \text{where}\quad  L_{s}(\tau):= \tau(\OO_{E_{v}}/\OO_{E_{v}[s]})  \subset (\vpi_{v}^{-s}\OO_{F_{v}}/\OO_{F_{v}})^{2},$$
is a well-defined $U_{v}$-orbit of lines in $(\vpi_{v}^{-s}\OO_{F_{v}}/\OO_{F_{v}})^{2}$.

If $U_{v}\supset U_{s, v}$, we may compare this with a corresponding invariant of an irreducible component $V(\lm)_{U}^{(C)}$ of $\X_{U, {\bf F}_{v}}$, defined as  
$$L(\lm):= U_{v} L_{s}(\lm), \qquad  \text{where}\quad  L_{s}(\lm) := \lm (\vpi_{v}^{-s}\OO_{F_{v}}/\OO_{F_{v}}) \subset (\vpi_{v}^{-s}\OO_{F_{v}}/\OO_{F_{v}})^{2}.$$

   \begin{defi*}  Let $V=V(\lm)_{U}^{(C)}\subset  \X_{U,{\bf F}_{v}} $ be an irreducible component. 
    The \emph{basin}
$$B_{V}$$
of $V$ is the set of CM points $z\in X_{F_{v}}^{\rm CM}$ such that either 
  \begin{itemize}
   \item $z$ has CM by $E$ with 
$E_{v}/F_{v}$ split, and the reduction of $z$ belongs to $V\cap \X_{{\bf F}_{v}}^{\rm ord}$, or 
\item $z$ has CM by $E$ with 
$E_{v}/F_{v}$  nonsplit, the  conductor $s$ of $z$ is such that  $U_{s,v}\subset U_{v}$, the reduction of $z$ belongs to the connected component $C$, and with notation as above
\beq \lb{Ln}
L(\tau)  = L(\lm).\eeq
      \end{itemize}
   \end{defi*}
A geometric description of basins is given in Proposition \ref{compare basins}.
  
\subsubsection{The main theorem} Let $X_{v}^{\an}$ be the  Berkovich analytic space attached to $X_{F_{v}}$. 
\begin{theoA} \lb{main thm} Let $(z_{n})$ be a sequence of points in $X_{F_{v}}^{\rm CM}$, and denote by the same name the image sequence in $X_{v}^{\an}$. Assume that the  $v$-conductor of $z_{n}$ tends to infinity.

The  sequence $(z_{n})$ has a limit if and only if it is eventually supported in  a single basin. If this is the case for the basin $\cB_{V}$ of the irreducible component  $V\subset \X_{{\bf F}_{v}}$, then 
$$\lim_{n\to \infty} z_{n}=\zeta_{V}$$
in $X_{v}^{\an}$.
\end{theoA}

It is perhaps remarkable  that, while the setup of the theorem only involves the Berkovich space $X_{v}^{\an}$,  the geometry  of the canonical integral model emerges to play an essential role in the conclusion. 

{\begin{figure}[h] \lb{fig1}
$$\includegraphics[width=.92\textwidth]{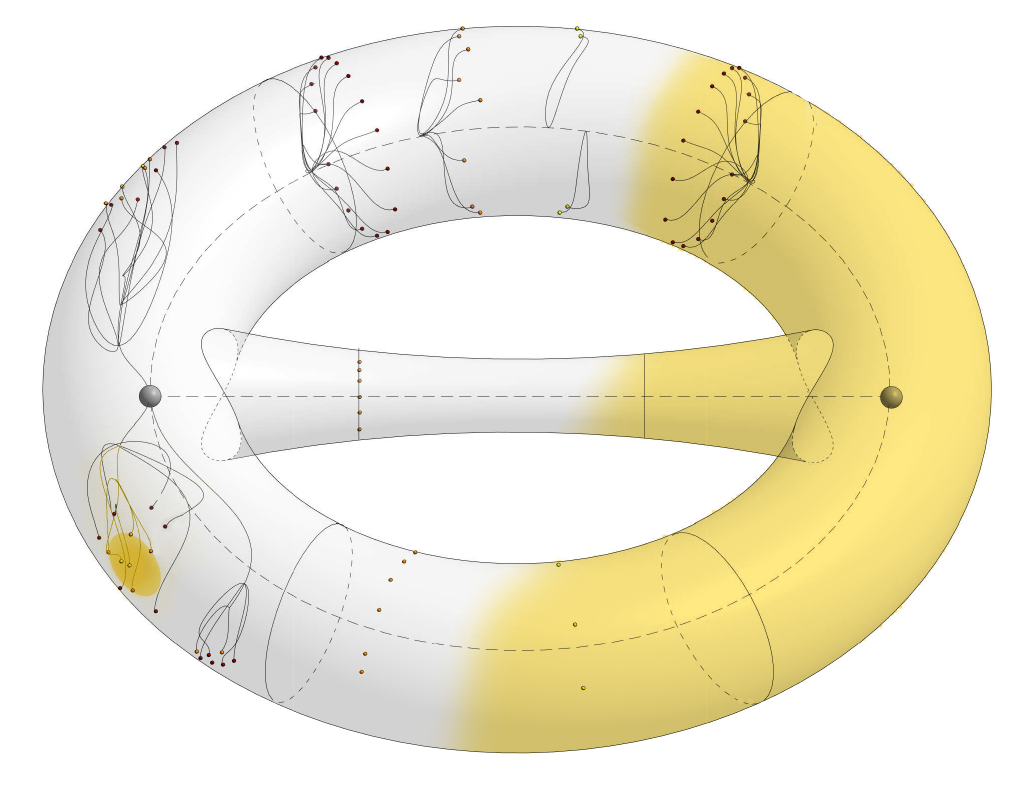} $$
\caption{{An illustration of Theorem \ref{main thm}, in the case where $X$ has genus~$2$, and $\X_{{\bf F}_{v}}$ consists of two rational irreducible components, intersecting in three supersingular points. The dual graph of $\X_{{\bf F}_{v}}$, which canonically embeds in $X_{v}^{\an}$ as a deformation retract \cite[\S 4]{ber}, is drawn  as  the dashed lines in the interior of a $2$-torus; the two thick points are the Shilov points of the components. The type-$1$ points of $X_{v}^{\an}$ are thought of as lying  on the surface of the $2$-torus (those with supersingular reduction are in the central part of the figure).
 In yellow, the complement of a typical open neighbourhood of the left  Shilov point (cf. Proposition \ref{fund sys}). 
In shades from yellow (lower conductor) to dark red (higher conductor), some CM points of $X_{F_{v}}$, drawn as Galois orbits of geometric points. }} \lb{fig1}
\end{figure} }

Theorem \ref{main thm} for the modular curve $X_{0}(1)$  was also independently, but earlier, obtained by Herrero--Menares--Rivera-Letelier \cite[Theorem A]{herrero}; see also \cite{Ric}.\footnote{The author conceived Theorem \ref{main thm} and the argument described in \S~\ref{sec 12}, in slightly different language, in the summer of 2018; soon thereafter he became aware of the 2016 thesis of S. Herrero, where \cite[Theorem A]{herrero} was first proved.}  From the point of view of the present paper (see the next subsection), for trivial level at $v$ the result is a relatively straightforward consequence of the theory of \cite{gross}, whereas more substantial work is needed to discern among the different Shilov points present in the case of higher level.

\subsection{Idea of proof, intersection formula, and organisation of the paper}\lb{sec 12} To illustrate the proof of Theorem \ref{main thm}, 
consider first a closed point  $a\in  |\X_{{\bf F}_{v}}|$  which is nonsingular in $\X_{{\bf F}_{{v}}}^{\rm red}$: this is the case if the $v$-level of $X$ is minimal ($U_{v}\cong \GL_{2}(\OO_{F_{v}})$) or $a$ is ordinary. Then ${\rm red}^{-1}(a)\subset X_{v}^{\an}$ is a (twisted) analytic open unit disc, say with coordinate $u$,  and Gross's theory of quasicanonical liftings \cite{gross} shows that CM points of $v$-conductor $s$ inside it lie in a circle $|u|=1-\eps_{s}$ with explicit $\eps_{s}\to 0$ as $s\to \infty$.
This immediately  implies that   any limit point of $(z_{n})$ is one of the points $\zeta_{V}$.  When the $v$-level of $X$ is minimal, irreducible and connected components coincide and this argument is enough to conclude.

The above description of the distribution of  CM points in nonsingular residue discs also implies the following: for each closed point $a\in|\X_{{\bf F}_{v}}|$, the  intersection multiplicity at $a$ between  $\X_{{\bf F}_{v}}$  and the Zariski closure $Z_{n}$ of those points $z_{n}$  of the sequence  that reduce to $a$ tends to infinity with $n$. 
As the intersection multiplicity between different irreducible components is bounded, this means that for  large $n$, exactly  one of the irreducible components of $\X_{{\bf F}_{v}}$ intersects $Z_{n}$ with high multiplicity -- equivalently, the corresponding Shilov point is close to $z_{n}$.  The argument described so far is developed in \S~\ref{sec2}.

\medskip

We thus need to compute the intersection multiplicity at $a\in |\X_{{\bf F}_{v}}|$ between irreducible components of $\X_{{\bf F}_{v}}$ and Zariski closures of CM points, the nontrivial case being when $a$ is supersingular.   Hence the rest  of   the paper is dedicated to proving a new formula for this multiplicity.

\medskip

We may state the formula using the  following notation:
\begin{itemize}
\item we let $q:= |{\bf F}_{v}|$, and denote by $|\cd |_{E_{v}}$ the absolute value on $E_{v}$ normalised by $|\vpi_{v}|=q^{-2}$;
\item we  choose any $\OO_{E_{v}}$-generator $\delta^{-1}\in \cD_{E_{v}}^{-1}:= \Hom_{\OO_{F_{v}}}(\OO_{E_{v}}, \OO_{F_{v}})$, and for $\tau\in \Hom_{F_{v}}(E_{v}, F_{v}^{2})$ we view the element  $\tau \delta\in \Hom_{\OO_{F_{v}}}(\OO_{E_{v}}, \OO_{F_{v}}^{2})\ot_{\OO_{E_{v}}} \cD_{E_{v}}\ot_{\OO_{F_{v}}}F_{v}=E_{v}^{2}$ as a column vector over $E_{v}$;
\item for $\lm \colon \OO_{F_{v}}\into \OO_{F_{v}}^{2}$, we denote by $\lm^{\perp}\colon \OO_{F_{v}}^{2}
\onto \OO_{F_{v}}$ any surjection such that $\Ker(\lm^{\perp})={\rm Im}(\lm)$; we view $\lm^{\perp}$ as a row vector in $F_{v}^{2}$;
\item if $\X'$ is any $2$-dimensional Noetherian formal scheme, and $x\in \X'$ is a regular closed point, we define the intersection multiplicity at $x$ of two irreducible subschemes $Z_{1}$, $Z_{2}$ intersecting properly (that is, such that the quantity below is finite) by
\beq\lb{int prop}
m_{x}(Z_{1}, Z_{2})={\rm length}_{\OO_{\X', x}} \OO_{\X', x}/(\cI_{Z_{1},x} + \cI_{Z_{2}, x})\eeq
where  $\cI_{Z_{i}}\subset \OO_{\X'}$ is the ideal sheaf defining $Z_{i}$.
\end{itemize}

\begin{theoA}\lb{thm B} Let $z\in X_{F_{v}}^{\rm CM}$ be a point with CM by $(E, \psi)$ of $v$-conductor $s$, such that $E_{v}/F_{v}$ is nonsplit. Let $Z$ be the Zariski closure of $z$ in $\X$, and let $\tau^{*}$
 be attached to $z$ as in \eqref{tau*}. Let $V=V(\lm)_{U}^{(C)}$  be an irreducible component through the reduction $\baar{z} \in \X_{{\bf F}_{v}} $ of $z$ (endowed with the reduced structure).

 Then
$${m_{\baar{z}}( Z, V) \over [\OO_{E_{v}}^{\ts}:\psi_{v}^{-1}(U_{v}) ]\cd [\OO_{F_{v}}^{\ts}:U_{v}\cap \OO_{F_{v}}^{\ts}]}
=
[{\bf F}_{v}(\baar{z}): {\bf F}_{v}]\cd 
c_{E_{v}} \cd  
q^{s }\cd
 \int_{U_{v}}  | \lm^{\perp}  g   \tau \delta  |_{E}^{-1}\, dg.$$
Here, the constant $c_{E_{v}}=1$ if $E_{v}/F_{v}$ is unramified and $c_{E_{v}}=1+q^{-1}$ if $E_{v}/F_{v} $ is ramified,
 and
   $dg$ is the Haar measure such that $\vol(\GL_{2}(\OO_{F_{v}}))=1$.
   \end{theoA}

 If we call  \emph{special cycles} in $\X$ those that are combinations of vertical components and  Zariski closures of CM points, then our formula  completes the calculations of intersections of special cycles in $\X$, where the case of intersections of cycles of the same type was treated respectively by Katz and Mazur (\cite[\S~13.8]{KM}, for vertical cycles) and, very recently,  Qirui Li (\cite{li}, for CM cycles). It will be clear that our method also allows to recover (and find the most general form of)  the Katz--Mazur intersection formula for vertical cycles.
 
 The intersection  problem in $\X$ is equivalent to one in a Lubin--Tate space $\cM_{U_{v}}$ of deformations of the  unique formal $\OO_{F_{v}}$-module $\cG_{F}$ of height $2 $ over $\baar{\bf F}_{v}$. We solve it by adapting the beautiful  method   devised by Li \cite{li} (and in turn inspired by a work of Weinstein \cite{weinstein}): after passing to infinite level in the Lubin--Tate tower, our intersection problem can be compared to an easier intersection problem in the formal group $\cG_{F}^{2}$.

The cycles in Lubin--Tate towers and formal groups are defined in \S\ref{sec4}. The computation of intersections and the completion of the proofs is in \S\ref{sec5}.

\subsubsection{Related work,  applications, and future directions} A corollary of Theorem \ref{main thm} is the following. Assume that the sequence $(z_{n})$ is convergent in $X_{v}^{\rm an}$, and let $\zeta_{V}$ be its limit. Then   the intersection multiplicity in $\X_{v}$ between a given horizontal divisor $H$ and the Zariski closures of  $z_{n}$  is eventually equal to  $m(H, V)$. A weak generalisation of this result is the key new ingredient in the recent proof of the $p$-adic Gross--Zagier formula at nonsplit primes \cite{nonsplit}, which was our main source of motivation.

Problems of equidistribution in Berkovich spaces have been considered in the context of arithmetic dynamics and Arakelov theory (see for example \cite{CL} and its  citation orbit). In fact,  sequences of CM points of increasing $p$-conductor may be obtained from the dynamics of some Hecke operator at $p$, for example the operator $U_{p}$ on a modular curve of Iwahori level at $p$ (more precisely, one should consider the forward-orbit of a given CM point).  As far as we know the phenomenon of  multiple, but finitely many, limit measures is exposed here for the first time, and it is a natural question to try and  understand it within a more general framework.

On the other hand, the $p$-adic dynamics of \emph{prime-to-$p$} Hecke operators on CM points, and the distribution of  CM points  whose $p$-conductors stay bounded,  have been recently been studied by Goren--Kassaei \cite{GK} 
and in the aforementioned  \cite{HMRL-k}, the latter work having some new arithmetic applications; it would be interesting to extend those results to our context.

Finally, it would also be very interesting to extend the intersection formula of Theorem \ref{thm B}  to higher-dimensional situations, and to explore its geometric implications and possible applications.

\subsection*{Acknowledgements} I would like to thank Qirui Li and Sebasti\'an Herrero for  correspondence on their respective works, and Christian Johansson and Michael Temkin for correspondence on $p$-adic geometry. I am also grateful  to Simone Dell'Ariccia for Figure \ref{fig1}, and to Francesco Maria Saettone and the referees for helpful comments.

\section{Reduction to intersection multiplicities}
\lb{sec2}
In this section, we prove Theorem \ref{main thm} in the case of minimal level, and reduce it to a statement on intersection multiplicities in general.

\subsection{Geometry of Berkovich curves}
The standard reference on Berkovich spaces is \cite{ber}, see especially \S\S~2-3.  

Fix for this subsection a discretely valued field $K$, with uniformiser $\vpi$,   ring of integers  $\OO_{K}$, and residue field $k$. We denote $q:= |\vpi|^{-1}$. 

\subsubsection{Generalities}
 Let $X $ be a compact Hausdorff strictly $K$-analytic Berkovich space  over $K$, generic fibre of a  topologically finitely presented  formal scheme $\X$  flat over $\Spf \OO_{K}$.  
 If $\X= \Spf A$ may be taken to be affine we say that $X$ is affinoid. In this case, denoting $A_{K}= A\ot_{\OO_{K}}K$,  the points   $x\in X$ are in canonical bijection with the multiplicative seminorms, denoted  $$|\cd(x)|\colon A_{K} \to \R, $$ 
extending the absolute value of $K$. The set of maximal ideals  ${\rm Max}\, A_{K}\subset \Spec A_{K}$ is embedded in $X$ via the map $x\mapsto [\phi\mapsto |\phi(x)|]$, where $|\cd| $ is the extension of $K$ to $K(x)$. The image of  ${\rm Max}\, A_{K}$ is called the set of type-$1$ points of $X$. It is dense in $X$. In general, if $X=\bigcup X_{i} $ is a cover of $X$ by affinoids, the set of type-$1$ points is the union of the sets of type-$1$ points of the $X_{i}$; it is well-defined, and dense in $X$.

\subsubsection{Reduction map and Shilov points}
There is a reduction map  to the special fibre of $\X$, 
$${\rm red}\colon X\to |\X_{k}|,$$
sending type-$1$ points to closed points. 
If $V\subset \X_{k}$ is an irreducible component
and $\xi_{V}\in \X_{k} $ is its generic point, then by \cite[Proposition 2.4.4]{ber} there is a unique point $\zeta_{V}\in X$ with  $${\rm red}(\zeta_{V})=\xi_{V}. $$
 We call $\zeta_{V}$ the \emph{Shilov point} of $V$.

\subsubsection{Type-1 points on curves as intersection multiplicities}
Suppose for the rest of this subsection that $X$ is a strictly $K$-analytic curve (that is $X$ is as above and $\dim X=1$), and that $\X$ is regular.  The following notation will be used throughout the paper.

\begin{enonce}{Notation} \lb{not Z} \emph{If $z_{?}$ is a type-$1$ point of $X$ (where `?' denotes an arbitrary decoration), 
we denote by  $$Z_{?}\subset \X$$ the closed formal subscheme defined as follows. Let  $A$ be topologically finitely presented flat $\OO_{K}$-algebra with an embedding $\Spf A\into \X$ such that $z_{?}$ corresponds to a point  $\frakm_{{?}}\in{\rm Max}\, A_{K}\subset X$; then $Z_{?}\subset \Spf A\subset \X$ is the image of $\Spf \OO_{K(z_{?})}$ via the map $A\to \OO_{K(z_{?})}$ deduced from $A_{K}\to A_{K}/\frakm_{{?}} =K(z_{?})$. We denote by $\baar{z}_{?}\in \X_{k}$ the image of the closed point of $\Spf \OO_{K(z)}$; by the definition of the reduction map in \cite[\S~2.4]{ber}, we have ${\rm red}(z_{?})=\baar{z}_{?}$.}

\emph{If $X_{0}$ is a proper smooth curve over $\Spec K$, $\X_{0}$ is a regular model over $\Spec \OO_{K}$, and $z_{0,?}$ is a closed point of $X_{0}$, we denote by $Z_{0,?}$ the Zariski closure of $z_{0,?}$ in $\X_{0}$, and by $\baar{z}_{0,?}$ the reduction of $z_{0,?}$ in $\X_{0, k}$. It is clear from the definitions  that  if $\X$ is the formal completion of $\X_{0}$ along the special fibre, $X=X_{0}^{\rm an}$, and $z_{?}$ is the type-$1$ point corresponding to $z_{0,?}$, then $Z_{?}$ is the formal completion of $Z_{0,?}$. For this reason we may think of $Z_{?}$ as the `Zariski closure'  of $z_{?}$.\footnote{This interpretation would acquire a more literal meaning if working in the context of adic spaces.}
}
\end{enonce}

\begin{lemm}\lb{seminorm int}
 Let $z\in X$ be a type-$1$ point, and assume that $Z\subset \X$ is normal. Let $\Spf A\subset \X$ be an affine neighbourhood of $Z$.
Then for every $\phi\in A$ we have
\beqq
|\phi(z)| = q^{-m(Z, \divisor(\phi))/ m(Z, \divisor(\vpi))},\eeqq
where in the right-hand side, $m:=m_{\baar{z}}$ is the intersection multiplicity at $\baar{z}$ as defined in \eqref{int prop}.
\end{lemm}
\begin{proof}
Write $Z=\Spf A/\frakp$ for some prime ideal $\frakp$, and let $\baar\frakm\supset \frakp$ be the maximal ideal of $A$ corresponding to $\baar{z}$.
After applying the function $-\log_{q}$, the left-hand side equals 
$${{\rm length}_{\OO_{K(z)}}\OO_{K(z)}/(\phi(z)) \over{\rm length}_{\OO_{K(z)}}\OO_{K(z)}/(\vpi)},$$ 
and the right-hand side equals 
$${{\rm length}_{A_{\baar\frakm}} A_{\baar\frakm}/(\frakp +(\phi)) \over {\rm length}_{A_{\baar\frakm}} {A_{\baar\frakm}}/(\frakp +(\vpi))}.$$ 
Now $A_{\baar\frakm}/\frakp\into K(z)$  is an integrally closed  $\OO_{K}$-algebra with fraction field $K(z)$; it follows that $A_{\baar\frakm}/\frakp\cong \OO_{K(z)}$, hence the two quantities above are equal.
\end{proof}

\subsubsection{Residue discs and neighbourhoods of Shilov points}
If  $a\in |\X_{k}|$ is a closed point, we denote by 
$$\Delta(a, 1- \eps),\qquad
 \eps>0, $$   an arbitrary increasing collection of compact subsets of ${\rm red}^{-1}(a)$ such that $\bigcup_{\eps > 0} \Delta(a, 1-\eps)= {\rm red}^{-1}(a)$. (The reader may keep in mind the case where  $a$ is a nonsingular point and $k$ is algebraically closed: then ${\rm red}^{-1}(a)$ is an open unit disc and one may take the  $\Delta(a,1- \eps)$  to be discs of radius $1-\eps$.)
 
 We define a similar collection of compact subsets of  $X$  indexed by the open points of $\X_{k}$. Thus let 
$V$ be an irreducible component  of $\X_{k}$  and let $\xi_{V}$ be its generic point.
 As $\X$ is regular, hence locally factorial, there is a finite  open cover by flat affine formal schemes 
  $\X=\bigcup_{\Y \in I} \Y$ such that  in  $\OO(\Y)$ we may factor
\beq\lb{factoriz}
\vpi = \prod_{V} \phi_{V\cap \Y}^{e_{V \cap \Y}} ,
\eeq
where    $\phi_{V\cap \Y}$ is a  generator of the (height-$1$ prime, or unit) ideal   $\frakp_{V\cap \Y}\subset\OO( \Y)$  and $e_{V\cap \Y}\geq 0$.
Let $Y\subset X$ denote  the Berkovich generic fibre of $\Y$.
We define 
\beq\lb{Dxiy}
\Delta(\xi_{V\cap \Y}, 1-\eps) &:=\{x\colon  |\phi^{e_{V\cap \Y}}_{V\cap \Y}(x)| \leq 1-\eps\}\subset Y \\
\Delta^{\circ}(\xi_{V\cap \Y}, 1-\eps) &:=\{x\colon  |\phi^{e_{V\cap \Y}}_{V\cap \Y}(x)| < 1-\eps\}\subset Y 
.\eeq
We denote by $e_{V}$ the multiplicity of $V$ in $\X_{k}$, that is, $e_{V}=e_{V\cap \Y}$ for any $\Y$ as above that is contained in the connected component of $V$ in $\X_{k}$. 

\begin{lemm} \lb{Dxi} For $\eps\in (0, 1)$ and $?\in \{\emptyset, \circ\}$,\footnote{Throughout this paper, we use the symbol `$\emptyset$' to denote `no symbol' when referring to pieces of notation.}  there  is a unique  subset 
$$\Delta^{?}(\xi_{V}, 1-\eps) \subset X$$
such that for all sufficiently small $\OO_{K}$-flat  affine open $\Y\subset \X$ with generic fibre $Y$, 
$$\Delta^{?}(\xi_{V}, 1-\eps)\cap Y =\Delta^{?}(\xi_{V\cap \Y}, 1-\eps),$$
where $\Delta^{?}(\xi_{V\cap \Y}, 1-\eps)$ is as in $\eqref{Dxiy}$.
Moreover:
\begin{enumerate}
\item if $z\in X$  reduces to a  point $a\in \X_{k} - V$, then  $z\notin \Delta(\xi_{V}, 1-\eps)$ for any $\eps>0$;
\item for $r<q^{-1}$, the sets  $\Delta(\xi_{V}, r)$ are empty;
\item the Shilov point  $\zeta_{V}$ belongs to $\Delta(\xi_{V}, q^{-1})$.
\end{enumerate}
\end{lemm}
\begin{proof} For the first part, it suffices to observe the trivial facts  that if $\Y'$ is a standard open affine subset of $\Y$ then the image in $\OO(\Y')$ of a  generator $\phi_{V\cap \Y}$ for $\frakp_{V\cap \Y}$ (respectively of the factorisation \eqref{factoriz}) is a generator for $\frakp_{V\cap \Y'}$ (respectively a factorisation of the same form). 

 For the additional statements, we may replace $X$ by any open affinoid $Y$ as above  such that $V\cap \Y\neq \emptyset$.  For the first statement, we note that under our assumptions we have $|\phi_{V}(z)|=1$ by definition of the map ${\rm red}$ in \cite[\S 2.4]{ber}. It follows from this and
 \eqref{factoriz} 
that $|\phi_{V}^{e_{V}}(\zeta_{V})| \leq |\vpi(\zeta_{V})|= q^{-1}$, proving the third statement.  The second statement is also immediate from \eqref{factoriz}.
\end{proof}

Denote by $|\X_{k}|^{0}\subset |\X_{k}|$ the set of closed points. 
\begin{prop} \lb{fund sys}
Let $X$ be a compact strictly $K$-analytic Berkovich curve, generic fibre of a regular  topologically finitely presented   formal scheme $\X$ flat over $\Spf \OO_{K}$. 
Let $V$ be an irreducible component of $\X_{k}$ and let  $\zeta_{V}\in X$ be the Shilov point of $V$. With notation as above, each of the systems of   open subsets of $X$
\beqq
 U( \zeta_{V}; A, \eps)&:= X- \bigcup_{a\in A}\Delta(a ,1-\eps), \qquad\qquad & A\subset |\X_{k}|-\{\xi_{V}\} \text{ finite}, \ \eps>0, \\
U'( \zeta_{V}; A', \eps)&:= \Delta^{\circ}(\xi_{V}, q^{-1}+\eps)- \bigcup_{a\in A'}\Delta(a ,1-\eps), \qquad\qquad & A'\subset |\X_{k}|^{0} \text{ finite}, \ \eps>0 ,
\eeqq
forms a fundamental system  of neighbourhoods of $\zeta_{V}$.
\end{prop}
\begin{proof} By construction, the intersection of all the open sets $U(\zeta_{V}; A, \eps) $ contains no element in ${\rm red}^{-1}(a)$ for $a\in |\X_{k}|$ a closed point. By the previous lemma, it contains $\zeta_{V}={\rm red}^{-1}(\xi_{V})$ but no element $\zeta_{V'}={\rm red}^{-1}(\xi_{V'})$ for $V'\neq V$.
\end{proof}

\begin{lemm} \lb{disjt}
The  open subsets $\Delta^{\circ}(\xi_{V}, q^{-1/2}) \subset X$  are pairwise disjoint as $V$ varies among irreducible components of $\X_{k}$. 
\end{lemm} 
\begin{proof}
Let $V=\divisor(\phi)$, $V'=\divisor(\phi')$ be distinct irreducible components of $\X_{k}$, and let $e_{V}$, $e_{V'}$ be the corresponding multiplicities in $\X_{k}$ (as defined after \eqref{Dxiy}). If $z\in \Delta^{\circ}(\xi_{V}, q^{-1/2}) \cap\Delta^{\circ}(\xi_{V'}, q^{-1/2}) $, we have
$$q^{-1}= |\vpi(z)| \leq   |\phi(z)|  |\phi'(z)| < q^{-1}, $$
a contradiction.
\end{proof}

\subsection{Reduction}\lb{sec 22}
In this subsection we  prove  a `soft' geometric version of our main theorem, in terms of geometric basins (certain open subsets of $X_{v}^{\an}$) rather than algebraic basins. In the case of minimal level $U_{v}=\GL_{2}(\OO_{F_{v}})$, this version is already equivalent to Theorem \ref{main thm}, as there is only one basin per connected component.   In general, the soft variant is vivid but not explicit; however, it reduces Theorem \ref{main thm} to the explicit comparison of algebraic and geometric basins, stated in Corollary  \ref{compare basins}.

We resume with the notation of the introduction. Thus let   $X_{v}^{\an}$ be the Berkovich  analytification of the Shimura curve $X_{F_{v}}=X_{U, F_{v}}$ as in the introduction, and let $\X$ be its regular model over $\Spf {\OO}_{F_{v}} $ defined in \cite{carayol}. 
We denote by $e$ the ramification degree of $E_{v}/F_{v}$, and for $K$ a local field with residue field $\kappa$ we denote  by $\zeta_{K}(s):= (1-|\kappa|^{-s})^{-1}$ its zeta function.

\begin{prop}
Let $(z_{n})_{n \in \N}$ be a sequence of CM points in $X_{v}^{\an}$ with $v$-conductor going to~$\infty$. The set of limit points of $(z_{n})$ is contained in the set of Shilov points of irreducible component of $\X_{{\bf F}_{v}}$.
\end{prop}

\begin{proof}
According to Proposition \ref{fund sys}, we need to show that for each closed point $a\in| \X_{{\bf F}_{v}}|$
and $\eps>0$, the sequence $z_{n}$ eventually leaves the compact set $\Delta(a, 1-\eps)$. 

Let $\breve{F}_{v}$ be the completion of the maximal unramified extension of $F_{v}$, and denote by  $\breve{X}_{v}^{\an}$ the analytification of $X_{\breve{F}_{v}}$. Its reduction is $\X_{k}$ where $k=\baar{\bf F}_{v}$. 
 
 As the transition maps in the tower $(X_{U})_{U}$ are finite and the images  of  CM points are CM points, it suffices to consider the case of minimal level $U_{v}\cong \GL_{2}(\OO_{F_{v}})$.  By the `Serre--Tate' theorem of \cite[\S 6.6]{carayol}, for each closed point $a\in |\X_{k}|$ the inverse image   ${\rm red}^{-1}(a)$ is canonically identified with the analytic generic fibre  of the  universal deformation space\footnote{More general versions of this object and of the Serre--Tate theorem will be precisely described in \S~\ref{LTs} below.} of the formal group $\cG_{a}$, which by \cite{Dr} is isomorphic to  the analytic open unit disc ${\bf D}$ generic fibre of ${\rm Spf}\breve{\OO}_{F ,v}\llb u\rrb$.
The  CM points in  ${\rm red}^{-1}(a)$ are explicitly described by  Gross's theory of quasicanonical liftings (see \cite{gross}, \cite{argos, argos-split}). 
Namely, if  $z_{n} \in X_{F_{v}}^{\rm CM}$ has   conductor $s=s(n)\geq 1$, then any point of $X_{\breve{F}_{v}}$ over  $z_{n}$ has residue field $ \breve{E}_{v,s}$,  a specific totally ramified abelian  (respectively dihedral if $v$ ramifies in $E$) extension of $\breve F_{v}$ of degree 
\beq\lb{vs}
m_{s(n)}:=e \zeta_{F_{v}}(1)\zeta_{E_{v}}(1)^{-1}q^{s(n)}.\eeq 
The inclusion of the point, spread out to the formal models,  is dual to a map $\breve{\OO}_{F_{v}}\llb u\rrb \to {\OO}_{{\breve E}_{v,s}}$ sending $u$ to a uniformiser, hence  $z_{n}$ lies in the circle $|u|= |\vpi|^{1/m_{s(n)}}$.
As $m_{s(n)}\to 0$ when $n\to \infty$,    any compact disc in ${\rm red}^{-1}(a)$ contains only finitely many CM points of $X_{\breve{F}_{v}}$. Therefore all points in the limit set of the sequence $(z_{n})$ in $X_{v}^{\an}$ reduce to a non-closed point of $\X_{{\bf F}_{v}}$. 
\end{proof}

\begin{rema} From the theory of quasicanonical liftings just discussed, it also follows that if $U_{v} = \GL_{2}(\OO_{F_{v}})$ and $z\in X_{v}^{\an}$ is a CM point of $v$-conductor $s$ with reduction $\baar{z}$, then
 \beq \lb{int qcan} m_{\baar{z}}(Z, \X_{{\bf F}_{v}}) /   [{\bf F}_{v}(\baar{z}) : {\bf F}_{v}]  = m_{s}=  e \zeta_{F_{v}}(1)\zeta_{E_{v}}(1)^{-1}q^{s}.\eeq
\end{rema}

We may now define the geometric counterpart of the basins defined in the introduction.

\begin{defi}\lb{bas geo} Let $V\subset \X_{{\bf F}_{v}}$ be an irreducible component. The \emph{geometric basin} of $V$ is the open subset
$$\mathscr{B}_{V}
 := \Delta^{\circ}(\xi_{V}, q^{-1/2}) \subset X_{v}^{\an}.$$
\end{defi}
By Lemma \ref{disjt}, the geometric basins are disjoint open subsets of $X_{v}^{\an}$. 

The following is the `soft' version of theorem \ref{main thm}. 
\begin{theo}\lb{soft} Let $(z_{n})_{n \in \N}$ be a sequence of CM points in $X_{v}^{\an}$, with $v$-conductor going to $\infty$.   The sequence has a limit if and only if it is eventually supported in a single \emph{geometric} basin $\mathscr{B}_{V}$. If this is the case for the geometric basin of the irreducible component $V\subset \X_{{\bf F}_{v}}$, then 
$$\lim_{m\to \infty} z_{n} = \zeta_{V}.$$
\end{theo}
\begin{proof}
By Proposition \ref{fund sys}, the  neighbourhoods  $\Delta^{\circ}(\xi_{V}, q^{-1}+\eps)$ of $\zeta_{V}$ are contained in the elements of  a fundamental system of neighbourhoods. Thus if the sequence $(z_{n})$ has a limit $\zeta_{V}$, it is eventually supported in $\Delta^{\circ}(\xi_{V}, q^{-1/2})= \mathscr{B}_{V}^{}$. Conversely,  if $z_{n}$ is  eventually supported in $\mathscr{B}_{V}^{}$, by Lemma \ref{disjt} it eventually leaves  $\mathscr{B}_{V'}^{}$
 for all $V'\neq V$, so that $\zeta_{V'}$ cannot be a limit point. Together with the previous result and the sequential compactness of $X_{v}^{\rm an}$, this shows that $z_{n}\to\zeta_{V}$.
\end{proof}

In view of Theorem \ref{soft}, Theorem \ref{main thm} is reduced to Corollary \ref{compare basins} below, a consequence of the following.

\begin{prop} \lb{bounded} 
Let $V$ be an irreducible component of $\X_{{\bf F}_{v}}$. 
 Let $\baar{z} \in \X_{{\bf F}_{v}}$ be a supersingular point and let $V'\neq V$ be  an irreducible component of $\X_{{\bf F}_{v}}$ through $\baar{z}$.
 The intersection multiplicity
$$m_{\baar{z}}(Z, V')$$
is bounded as $z$ varies among CM points of $X_{F_{v}}$ in $\cB_{V}$ reducing to $\baar{z}$.  
\end{prop}
The proof of Proposition \ref{bounded}, given at the very end of the paper (after Corollary \ref{cor bounded}), will be a corollary of the formula of Theorem \ref{thm B}, which computes geometric intersection multiplicities in terms of algebraic data.

\begin{coro} \lb{compare basins} 
Let $V\subset \X_{{\bf F}_{v}}$ be an irreducible component. For all $z\in X_{v}^{\rm CM}$ that is either of  ordinary reduction, or  of sufficiently high  conductor  (depending only on the level $U_{v}$), we have
$$z\in B_{V} \Longleftrightarrow z \in \mathscr{B}_{V}.$$
\end{coro}
\begin{proof}[{Proof of Corollary \ref{compare basins} and of Theorem \ref{main thm}, assuming Proposition  \ref{bounded}.}] The corollary is obvious for points of ordinary reduction.
Let $z\in B_{V}$ have supersingular reduction and conductor $s$. We have
$$e_{V}\cd m(Z, V) / m(Z, \X_{{\bf F}_{v}})  = 1 - \sum_{V'\neq V}  e_{V'} m(Z, V')/m(Z, \X_{{\bf F}_{v}}).$$
 By  \eqref{int qcan} (which remains valid for $\X$ of arbitrary level by the projection formula) and Proposition \ref{bounded}, the right-hand side is close to~$1$ if $s$ is sufficiently large. In particular, by Lemma \ref{seminorm int} and the definition, the point $z$ belongs to $\mathscr{B}_{V}$. Since the $B_{V}$ form a complete partition of the set of CM points of sufficiently high conductor, this also proves the reverse implication.  \end{proof}
 
 The following basic result will be established at the end of \S~\ref{sec 33}, based on a modular interpretation of CM curves.
 \begin{lemm}\lb{is reg} For any $z\in X_{F_{v}}^{\rm CM}$, its Zariski closure of $Z\subset \X$ is a regular curve.
 \end{lemm}
 
 \begin{proof}[{Proof of Corollary \ref{compare basins} and of Theorem \ref{main thm}, assuming  Lemma \ref{is reg} and Proposition  \ref{bounded}.}] The corollary is obvious for points of ordinary reduction.
Let $z\in B_{V}$ have supersingular reduction and conductor $s$. We have
$$e_{V}\cd m(Z, V) / m(Z, \X_{{\bf F}_{v}})  = 1 - \sum_{V'\neq V}  e_{V'} m(Z, V')/m(Z, \X_{{\bf F}_{v}}).$$
 By  \eqref{int qcan} (which remains valid for $\X$ of arbitrary level by the projection formula) and Proposition \ref{bounded}, the right-hand side is close to~$1$ if $s$ is sufficiently large. Therefore, by Lemma \ref{is reg}, Lemma \ref{seminorm int} (applied to the formal completion of $\X$ along the special fibre),  and the definitions, the point $z$ belongs to $\mathscr{B}_{V}$. Since the $B_{V}$ form a complete partition of the set of CM points of sufficiently high conductor, this also proves the reverse implication of  Corollary \ref{compare basins}. As noted above, Theorem \ref{main thm} is then implied by  Theorem \ref{soft}.
  \end{proof}

\section{Special cycles in Lubin--Tate spaces and  in formal groups}  \lb{sec4}
The goal of the rest of the paper is to establish an intersection multiplicity formula that will imply Proposition  \ref{bounded}. The formal completion  of ${\X}$ at a supersingular geometric  point admits a purely local description as a Lubin--Tate space $\cM$. After introducing some notation, in this section we define $\cM$ and its collection of special curves, comparing those with the special curves in $\X$. 
 Then we define  companion cycles in  formal groups; the definitions are a special case of those in \cite{li} for CM curves, and are new for vertical curves.

\subsection{Notation} The notation introduced here will be used throughout the rest of the paper unless otherwise noted. 
\subsubsection{Valued fields}
We fix a prime $v$ of $F$ nonsplit in $E$   and work  in a purely local setting, dropping the subscript $v$ from the notation; thus $F$ and $E$ will denote the local fields previously denoted by $F_{v}$ and $E_{v}$ respectively.   We denote by $\vpi$ a fixed uniformiser in $F$, by $v$ the valuation on $F^{\ts}$ (so $v(\vpi)=1$), and  by  $ |\cd |:=q^{-v(\cd)}$ the standard absolute value on $F$.

If $K$ is a finite extension of $F$, we denote by $\breve K$ the completion of the maximal unramified extension of $K$, by $\breve{\OO}_{K}$ its ring of integers, and by $k$ the residue field of $\breve{K}$ (which is independent of the choice of $K$).  

\subsubsection{Functors on $\OO_{F}$-modules}
We denote by  
$$M_{[n]}:= M/\vpi^{n} M$$
the truncation of $\OO_{F}$-modules, and by
$$M^{*}:= \Hom_{\OO_{F}} (M, F/\OO_{F}), \qquad M^{\vee}:= \Hom_{\OO_{F}} (M, \OO_{F})$$
the Pontryagin and, respectively, plain dualities on $\OO_{F}$-modules.  
We stipulate the following convention on the order of reading  of symbols:
$$\OO_{K, [N]}^{*}:= (\OO_{K, [N]})^{*}, \qquad \OO_{K, [N]}^{\vee}:= (\OO_{K}^{\vee})_{[N]}.$$

We let 
$$\cD_{K}^{-1}:= \OO_{K}^{\vee } = \Hom_{\OO_{F}}(\OO_{K} , \OO_{F})$$
denote the relative inverse different of $K/F$, an invertible $\OO_{K}$-module.

\subsubsection{Integers} 
 For compact open subgroups $C_{K}\subset K^{\ts}$,
we define integers
\beq \lb{dck} d_{C_{K}}= [\OO_{K}^{\ts}:C_{K}].\eeq
 If $E/F$ is a quadratic extension, we also define 
  $ d_{0}=d_{\Ig, 0}=d_{E, 0}=1$  and, for $n\geq 1$,
    $$d_{n} := q^{2n, \quad }d_{E, n}= \zeta_{E}(1)^{-1}q^{2n},  \quad d_{\Ig,n}:= \zeta_{F}(1)^{-1}q^{2n}.$$

\subsubsection{Miscellaneaous}
We denote by  $\cC_{K}$ the category of complete Noetherian local $\breve{\OO}_{K}$-algebras with residue field $k$, and by $\cC_{k}$ its subcategory consisting of objects that are $k$-algebras.

Finally, we will  freely use the language and notation of the Appendix, which the reader is now invited to skim through.

\subsection{Lubin--Tate spaces}\lb{LTs}  For basic references, see \cite{Dr}, or \cite[Appendice]{carayol}.

Let $K$ be a finite extension of $F$.  Let $\cG$ be a formal ${\OO}_{K}$-module of dimension 1 and height $h$ over an object $A$ of $\cC_{F}$.\footnote{Note as in \cite[\S~2.2]{li} that the $\OO_{K}$-action on ${\rm Lie} \, \cG\cong A$ gives $A$ an $\OO_{K}$-algebra structure, so that $A$ is also an object of $\cC_{K}$.}
A \emph{Drinfeld  structure} on $\cG$ of level $n$
 is an $\OO_{K}$-module map 
$$\alpha\colon \OO_{K, [n]}^{h, *}\to \Hom_{\Spf A}(\Spf A, \cG)$$ 
such that
 $$\alpha_{*}  (\OO_{K, [n]}^{h, *}) :=\sum_{x\in \OO_{K,[n]}^{h, *} } [\alpha(x)] =\cG[\vpi^{n}]$$ 
 as Cartier divisors.   A  Drinfeld structure of infinite level  on $\cG$ is a map $\OO_{K}^{h,*}\to \cG(A)=\Hom_{\Spf A}(\Spf A, \cG)$ whose restriction to $\OO_{K, [n]}^{h, *}$ is a Drinfeld structure of level $n$ for all $n$.

Let $\cG_{h,K}$ be a formal $\OO_{K}$-module of height $h$ over $k$, which is unique up to isomorphism. Consider the  functor $\cM_{h, K, n}$ on $\cC_{F} $ that associates to an object $A$ the set of equivalence classes of triples 
\beq\lb{triples}
[\cG, \iota, \alpha],\eeq
 where $\cG$ is a formal $\OO_{K}$-module over $A$ of height $h$, 
$\iota\colon \cG_{h,K} \to \cG \ts_{\Spf A}\Spf k$ is a quasi-isogeny of height~ $0$, and
$\alpha\colon \OO_{K, [n]}^{h, *} \to \cG(A)$ is a Drinfeld structure of level $n$. A theorem of Drinfeld  \cite{Dr}
asserts that $\cM_{h, K, n}$ is representable by a regular formal scheme finite flat over $\cM_{h, K, 0}$,
and  that $ \cM_{h, K,0}\cong_{/\breve{\OO}_{F}}\Spf \breve{\OO}_{K}\llb u_{1}, \ldots, u_{h-1}\rrb$.  

If $U\subset\GL_{h}(\OO_{K})$ is the open compact subgroup 
$$U=U^{(h, K)}_{n}:=\Ker (\GL_{h}(\OO_{K})\to \GL_{h}(\OO_{K}/\vpi^{n})),\footnote{We will simply write $U_{n}$ for $U_{n}^{(h, K)}$ when $h$, $K$ are clear from context.}$$
 we define $\cM_{h, K,U}:= \cM_{h, K,n} $. If  $U\subset\GL_{h}(\OO_{K})$ is any compact subgroup containing $U_{n}$, we define $\cM_{h, K,U}:= \cM_{h, K, n}/(U/U_{n})$. 

The formal scheme (representing the functors) $\cM_{h, K, U}$  are called \emph{Lubin--Tate spaces}. 
In what follows, we will denote 
\beq \lb{GEF}
\cG_{E}:= \cG_{1, K} , \qquad \cG_{F}:=\cG_{2,F}\eeq
and identify $\cG_{F}$ with the  image of $\cG_{E}$ under the forgetful functor. The endomorphism ring of $\cG_{E}$ is $\OO_{E}$; the endomorphism ring of $\cG_{F}$ is the ring of integers (or elements with integral norms) of the unique division algebra $D$ of rank 4 over $F$. 

We will also denote 
 $$\cM_{U}:= \cM_{2, F, U},\ \cM_{n}:= \cM_{{U}^{(2, F)}_{n}}, \qquad \qquad \cN_{C}:= \cM_{1, E, C},\  \cN_{n}:= \cN_{U^{(1 , E)}_{n}}.$$
The space $\cN_{C}$ is isomorphic to $\Spf \breve{\OO}_{E,C}$, the ring of integers in the abelian extension of $\breve{E}$ with Galois group $\OO_{E}^{\ts}/C$.  

Finally, we denote by 
$$\cM_{\infty}:= \varprojlim_{n}\cM_{n}, \qquad \cN_{\infty}:= \varprojlim_{n}\cN_{n},$$
 as pro-objects of a category $\cF$ of Noetherian formal schemes (cf. the Appendix), and as the functors on $\cC_{F}$ of deformations of $\cG_{F}$ (respectively $\cG_{E}$) together with a Drinfeld structure of infinite level.

We fix for the rest of the paper a level $U\subset \GL_{2}(\OO_{F})$ and write $\cM:=\cM_{U}$. We denote by $M$ the Berkovich generic fibre of $\cM$. Let $\cG\to \X_{U}$ be Carayol's sheaf of divisible $\OO_{F_{v}}$-modules recalled in \S~\ref{sec 11}, and for a $z\in \X_{U}$, denote by $\cG_{z}$ the fibre of $\cG$ at $z$.
 
\paragraph{A Serre--Tate theorem}
We restore the subscript $v$ just for the following proposition.
\begin{prop}[Carayol  {\cite[\S7.4]{carayol}}] \lb{serretate} Let $\baar{z}\in \X_{U,k} $ be a closed point.
\begin{enumerate}
\item
If $\baar{z} $ is supersingular, then, after identifying   $\cG_{F}=\cG_{\baar{z}}$, the formal completion of $\X_{\breve{\OO}_{F_{v}}} $ at  ${\baar{z}}$ 
 is canonically  $\breve{\OO}_{F_{v}}$-isomorphic to the Lubin--Tate space $\cM_{U_{v}}$, via the map sending an $A$-valued point $z$ to $\cG_{z}$ with its level structure $\alpha_{z}$. 
\item 
 If $\baar{z} $ is ordinary,  the formal completion of $\X_{\breve{\OO}_{F_{v}}} $ at  $  {\baar{z}}$ 
  is  canonically $\breve{\OO}_{F_{v}}$-isomorphic to the deformation space of the $\OO_{F_{v}}$-divisible module $\cG_{F_{v}}^{(1)}\oplus F_{v}/\OO_{F_{v}}$ over $k$, and non-canonically to $\Spf \breve{\OO}_{F_{v}}\llb u\rrb$; in particular, it is smooth over $\breve{\OO}_{F_{v}}$
\end{enumerate}
\end{prop}

\subsection{Special curves in  Lubin--Tate spaces}\lb{sec 33} We consider two special classes of  curves in $\cM$, the irreducible components of the special fibre, that we call \emph{vertical curves},  and the ``Zariski closures'' of CM points, which we call \emph{CM curves}.  Then we compare them with the corresponding objects in the Shimura curve $\X$.

\subsubsection{Vertical curves}  Let $A $ be a complete Noetherian $k$-algebra and let $\cG$ be a formal $\OO_{F}$-module of height 2 over $A$. An \emph{Igusa structure} of level $n$ on $\cG$ is an $\OO_{F}$-module map 
$$\gamma\colon \OO_{F,[n]}^{*}\to \Hom_{\Spf A }(\Spf A, \cG)$$
such that $q^{n}\cd \gamma_{*} ( \OO_{F,[n]}^{*})= \cG[\vpi^{n}]$ as Cartier divisors.\footnote{Our $\Ig_{n}$ should be compared with the `exotic Igusa curve' ${\rm ExIg}(p^{n},n) $ of \cite[(12.10.5.1)]{KM}; cf. \emph{ibid.} Proposition 13.7.5.}
 The functor  on $\cC_{k}$ sending an object $A$ to the set of isomorphism classes of triples $[\cG, \iota, \gamma]$, where $[\cG, \iota]$ is a deformation of $\cG_{F}$ as in \eqref{triples},  and $\gamma$ is an Igusa structure of level $n$, is representable by a formally smooth curve $$\Ig_{n}\to \Spec k$$
called the $n^{\rm th}$ Igusa curve. It is of degree $d_{\rm Ig,n} $  over $\Ig_{0}= \cM_{U_{0}, k}\cong \Spf k\llb u\rrb$. We define  $\Ig_{\infty}:=\varprojlim_{n}\Ig_{n}$ and 
$$ [{\rm Ig }_{\infty}]^{\circ}:=\varprojlim  \,  [\Ig_{n}]^{\circ}\in K_{0}(\Ig_{\infty})_{\Q}, \qquad  [\Ig_{n}]^{\circ}:=  d_{\Ig, n}^{-1}\cd  [\Ig_{n}].$$

 Let  
$$\lm \in {\bf P}^{1}(F) \subset \OO_{F}^{\ts}\bks \Hom_{\OO_{F}}( \OO_{F},  \OO_{F}^{2}).$$
 We attach to $\lm$ the morphism 
\beqq
f_{\lm}\colon \Ig_{\infty} &\to \cM_{U,k}\into \cM_{U}\\
[\cG, \iota, \gamma] & \mapsto [\cG, \iota, \gamma\circ\lm^{*}]
\eeqq
(where $ \cM_{U,k}\into \cM_{U}$ is the natural map base-change of $k\into \Spf \breve{\OO}_{F}$)
 and the irreducible and reduced curve
\beq
f_{\lm}(\Ig_{\infty})=V({\lm})_{U} \subset \cM_{U, k}
\eeq
defined by 
the closed condition that any extension to $\OO_{F}^{*,2}$ of  the level structure $\alpha$ of a geometric point  $[\cG, \iota, \alpha]$ belonging to $V(\lm)_{U}$ should   factor through $\lm^{*} g\colon \OO_{F}^{*,2} \to \OO_{F}^{*}$ for some $g\in U$.
 The  ${V}({\lm})_{U}$, for $\lm \in {\bf P}^{1}(F)/U$,  are   the  irreducible components of $\cM_{U,k}$ and they all meet at the closed point \cite[Appendice \S~8, and \S~9.4]{carayol}.

\begin{defi} The \emph{vertical curve} attached to $\lm $  in $\cM_{\infty}$, respectively $\cM_{U}$, is the curve
$$V(\lm):=f_{\lm }(\Ig_{\infty})\subset \cM_{\infty}, \qquad V(\lm)_{U}:=f_{\lm , U}(\Ig_{\infty})\subset \cM_{U}.$$

The corresponding normalised \emph{vertical cycle}  in $K_{0}(\cM_{\infty})_{\Q}$, respectively $K_{0}(\cM_{U})_{\Q}  $), is 
$$[V(\lm)]^{\circ}:=f_{\lm , *}[\Ig_{\infty}]^{\circ} = \varprojlim_{U}  [V(\lm)_{U}]^{\circ} , \qquad [V(\lm)_{U}]^{\circ} = d_{U\cap \OO_{F}^{\ts}}^{-1}
\cd  [V(\lm)_{U}].
$$
\end{defi}
\subsubsection{CM curves} We consider a local  analogue of CM points and of their Zariski closures in integral models. These will be images of maps of Lubin--Tate towers and we start by defining the data parametrising them. The definition is slightly streamlined version of the one in  \cite[\S~2.5]{li}.

Consider an element
\beqq\lb{qisog}
(\vphi, \tau)\in {\rm LT}(E, F):= {\rm QIsog}(\cG_{E}, \cG_{F}) \ts_{E^{\ts}} {\rm Isom}_{F}(E, F^{2}) \eeqq
where the notation means that  $\vphi$ is a quasi-isogeny  of formal  $\OO_{F}$-modules,  $\tau$ is an $F$-linear isomorphism, and we consider the quotient of the product by the relation $(\vphi\circ t, \tau \circ t)\sim (\vphi, \tau)$, $t\in E^{\ts}$.  
  There exists a shortest  integer interval  $[r, s]$  such that $\vpi^{-r} \OO_{F}^{2}\subset \tau(\OO_{E})\subset \vpi^{-s_{}}\OO_{F}^{2}$; this depends on the choice of representative $\tau$. 

We define two   functions, \emph{conductor} and  \emph{height}, on ${\rm LT}(E, F)$ by 
\beq\lb{ht tau}
c(\vphi, \tau) &:= s-r,\\
 {\rm ht} (\vphi, \tau) &:=    -{\rm ht}(\vphi) + \log_{q}[\tau(\OO_{E}): \OO_{F}^{2} ],
\eeq
where ${\rm ht}(\vphi)$ is the usual $\OO_{F}$-height of a quasi-isogeny of $\OO_{F}$-modules, and $[ \, : \, ]$ denotes generalised index (that is, for two lattices $\Lm$, $\Lm'\subset F^{2}$, we have $[\Lm: \Lm']:={|\Lm/\Lm\cap \Lm'|\over |\Lm'/\Lm\cap \Lm'|}$).  We denote by ${\rm LT}(E, F)^{ 0}$ the subspace of pairs of height~$0$.

 \emph{From now on, the  representatives $(\vphi, \tau)$ of elements of ${\rm LT}(E, F)^{0}$  will always be chosen to satisfy  $r_{}=0$.}\footnote{The reader unwilling to make this assumption will simply need to interpret some maps of formal $\OO_{F}$-modules  in the diagrams to follow as quasi-isogenies.} In this case $c(\vphi, \tau) = s= {\rm ht}(\vphi)$.

To an element   $(\vphi, \tau)\in {\rm LT}(E, F)^{0}$ we attach a morphism $f_{(\vphi, \tau)}$ of Lubin--Tate towers, given by the following morphism of functors on $\cC_{F}$
(which are representable when restricted to finite levels):
\beq \lb{def f}
f_{(\vphi, \tau)}\colon \cN_{\infty} &\to  \cM_{\infty}\\
[\cG, \iota, \beta]&\mapsto [\cG/\beta_{*}K, \baar{\vphi}_{\beta}\circ\iota\circ 
 \vphi^{-1}, \beta'\circ \tau^{*}],
\eeq
where the notation is according to the following commutative diagrams:

\begin{equation}
\lb{CM diag}
\xymatrix{
& K:= (\OO_{E}/\tau^{-1}\OO_{F}^{2})^{*}  \ar[d] \ar[r]	 &\beta_{*}K \ar[d]	 &		\\	
	& \OO_{E}^{*}\ar[d] \ar^{\beta}[r] &   \cG \ar^{\vphi_{\beta}}[d]& \baar{\cG} \ar^{\baar{\vphi}_{\beta}}[d] &\ar[l]_{\iota} \cG_{E}  \ar^{\vphi}[r]& \cG_{F}\ar@{..>}[lld]\\
\OO_{F}^{2, *} \ar^{\tau^{*}}[r]   & \tau^{*}(\OO_{F}^{2, *}) \ar^{\beta'}[r] & \cG/\beta_{*}K&  \baar{\cG}/\baar{\beta}_{*}K.
}
\end{equation}
In the above diagrams, the columns are exact, the bars denote reduction modulo the maximal ideal of $\breve{\OO}_{E}$, and $\beta_{*}K $ is the  subgroup scheme  of $ \cG $ defined by the product of ideal sheaves of the $\beta_{*}\kappa$, $\kappa\in K$.

Let $\psi\colon E\into M_{2}(F)$ be the $F$-algebra morphism such that $\psi(t)\tau(x)= \tau(tx)$ for all $t, x\in E$. The composition to level $U$ factors as  
\beq 
\lb{def fU}
f_{(\vphi, \tau), U}:={\rm p}_{U}\circ f_{(\vphi, \tau)} \colon  \cN_{\infty}\to \cN_{C}\to \cM_{U} \eeq
 where $C:=\psi^{-1}(U)$ and the second map is a finite monomorphism, hence a closed immersion. In particular, the morphism of functors defined in infinite level does define  a morphism in the pro-category of Noetherian formal schemes $\widehat{\mathscr{F}}$ defined in the Appendix.

   If $U=U_{0}$ and $(\vphi, \tau)$ has conductor $s$,  then $C=(\OO_{F}+\vpi^{s}\OO_{E})^{\ts}$, and the image of $\cN_{C}$ is a quasicanonical lift of level~$s$. 

\begin{defi}
The \emph{CM curve}  in $\cM_{\infty}$ (respectively $\cM_{U}$) attached to $(\vphi, \tau)$ is 
$$Z(\vphi, \tau):= f_{(\vphi, \tau)}(\cN_{\infty})\subset \cM_{\infty}, \qquad Z(\vphi, \tau)_{U}:= f_{(\vphi, \tau), U}(\cN_{\infty})\subset \cM_{U}. $$
where $ f_{(\vphi, \tau)}$ (respectively $ f_{(\vphi, \tau), U}$) is the morphism defined by \eqref{def f} (respectively \eqref{def fU}).

The local CM point $z(\vphi, \tau)$ in the Berkovich generic fibre  $M$ of $\cM$ is the generic fibre of $Z(\vphi, \tau)$; equivalently, it is the image under $f_{(\vphi, \tau)} $ of the single point in the generic fibre of  $\cN_{n}$ for any sufficiently large $n$; thus the notation is consistent with Notation \ref{not Z}.

The normalised  \emph{CM cycle}   attached to $(\vphi, \tau)$ in $K_{0}(\cM_{\infty})_{\Q}$  (respectively $K_{0}(\cM_{U})_{\Q}  $) is 
$$[Z(\vphi, \tau)]^{\circ}:= f_{(\vphi, \tau), *} [\cN_{\infty}]^{\circ}= \varprojlim_{U}  \,  [Z(\vphi, \tau)_{U}]^{\circ}, \qquad  
[Z(\vphi, \tau)_{U}]^{\circ} = {1\over d_{\psi^{-1}(U)}}\cd  [Z(\vphi, \tau)_{U}].$$
\end{defi}

\begin{defi} \lb{basin M}
Let $V=V(\lm)_{U}\subset \cM_{k}$ be the irreducible component parametrised by the $U$-class of $\lm$, and let $n$ be minimal such that $U\supset U_{n}$. 
The \emph{basin} $B_{V(\lm)_{U}}$ of $V$ is the set of those CM points   $z=z(\vphi, \tau)\in M$ of conductor $s\geq n$, such that if $(\vphi,\tau)$ is chosen to satisfy $\OO_{F}^{2}\subset \tau(\OO_{E})\subset \vpi^{-s}\OO_{F}^{2}$, the lines
\beq\lb{Ldef}
L_{s}(\tau) := 
\tau(\OO_{E})/\OO_{F}^{2}, \qquad
 L_{s}(\lm):= \lm(\vpi^{-s}\OO_{F}/\OO_{F}) 
\qquad \subset \quad  (\vpi^{-s}\OO_{F}/ \OO_{F})^{2}
\eeq
are in  the same orbit for the left action of $U$. 
\end{defi}

\subsubsection{Comparison of  local and  global objects}
We temporarily restore the subscripts $v$ for local objects. We use the notation of the introduction for CM points on our fixed Shimura curve $X=X_{U}$, and we abbreviate $\cM=\cM_{U_{v}}$, as well as $M=M_{U_{v}}$ for the generic fibre of $\cM$. 
\begin{lemm}\lb{compare basins 2} Let   $z\in X^{\rm CM}_{F_{v}}$ be a point with complex multiplication by $(E, \psi)$, with $E_{v}/F_{v}$ nonsplit. Let $Z$ be the Zariski closure of $z$ in $\X$.  Let $\tau\colon E_{v}\to F_{v}^{2}$ be such that    $\psi_{v}(t)\tau(x)= \tau(tx)$ for all $t, x\in E_{v}$.
Let $\baar{z}\in \X_{{\bf F}_{v}}$ be the reduction of $z$ and let $C\subset  \X_{{\bf F}_{v}}$ be the connected component containing $\baar{z}$.   Let   $\lm \in {\bf P}^{1}(F_{v})/U_{v}$, let 
$$V=V(\lm)_{U}^{(C)}\subset C\subset \X_{{\bf F}_{v}} $$ be the corresponding irreducible component.

Let $\baar{z}' \in \X_{k}$ be any point over $\baar{z}$. Let $z'\in M$  be any point mapping to $z\in X_{v}^{\rm an}$ under the identification of $M $ with  ${\rm red}^{-1}(\baar{z})\subset X^{\rm an}_{\breve{F}_{v}}$ given by Proposition \ref{serretate}. 
 Let $V'\subset \cM_{k}$ be any irreducible component  mapping to $V\subset \X_{{\bf F}_{v}}$. 

\begin{enumerate}
\item $V'= V(\lm)_{U_{v}}\subset\cM_{U_{v}, k}$.
\item There exist $(\vphi, \tau)\in {\rm LT}(E, F)^{0}$, whose second component is equal (up to rescaling by $E_{v}^{\ts}$) to the isomorphism $\tau\colon E_{v}\to F_{v}^{2}$ attached to $z$, such that $z'=z(\vphi, \tau)$. 
\item The CM point $z$ belongs to $\cB_{V}$ if and only if $z'$ belongs to $\cB_{V'}$.
\item The images of $Z':= Z(\vphi, \tau)$ and of $Z$ in the formal completion $\hat{\X}_{\baar z}$ of  $\X$ at $\baar z$ (via the map of Proposition \ref{serretate} and the completion, respectively) are equal.
\item We have the following relation between intersection multiplicities in $\X$ and $\cM_{U_{v}}$:
$$m_{\baar{z}}(Z, V)= [{\bf F}_{v}(\baar{z}): {\bf F}_{v}]\cd  m(Z', V').$$
\end{enumerate}
\end{lemm}
\begin{proof} The first thee statements are clear from the definitions. For the fourth statement, let $\Spf A\subset \hat{\X}_{\baar{z}}$ be a neighbourhood of $\baar{z}$ with $A$ a topologically finitely presented flat $\OO_{F_{v}}$-algebra.  Then the images of $Z$ and $Z'$ are both induced by the unique map $A\to \OO_{F_{v}(z)}$ lifting the map $A_{F_{v}}\to F_{v}(z)$  corresponding to~$z$.

 Finally for the last statement, observe that the right hand side is the sum of the intersection multiplicities of all the  $\Gal({\bf F}_{v}(\baar{z})/{\bf F}_{v})$-conjugates of the pair $(Z', V')$ (as a pair of cycles in the completion of $\X_{\breve{\OO}_{F_{v}}}$ at $\baar{z}'$), and all those multiplicities are equal. 
\end{proof}

\begin{proof}[Proof of Lemma \ref{is reg}] By Lemma \ref{compare basins 2}, it suffices to prove the regularity of the formal scheme $Z'=Z(\vphi, \tau)_{U_{v}}$; using the notation of \eqref{def fU}, this follows from the fact that $Z'$ is the image of the regular formal scheme $\cN_{C}$  under a closed immersion.
\end{proof}

\subsection{Special cycles in formal groups and  their Tate modules}  We define special  cycles, analogous to those defined above in $\cM_{\infty}$,  inside  $\cG_{F}^{2}$, its truncations or  Tate modules.  The definitions are variants of those of \cite[\S~3.1, 3.3]{li}.
Later in \S\ref{41}, we will reduce the intersection problem in Lubin--Tate spaces  to a similar problem in $\cG_{F}^{2}$.

Let $K=F$ or  $E$ and let $\cG_{K}= \eqref{GEF}$.
  For
    $?\in \{ \emptyset, [N]\}$ and $\prime\in\{* , \vee\}$, let 
\beq
{}'\cG_{K, ?}:=\Hom_{\OO_{K}} (\OO_{K, ?}^{\prime}, \cG_{K})
\eeq
Then we have  identifications and maps
\beq\lb{jN}
\xymatrix{
{}^{*}\cG_{K}\cong\cD_{K}^{-1}\ot_{\OO_{K}}T\cG_{K}^{} 
\ar@{->>}[d]^{\pi_{N}} & \\
 {}^{*}\cG_{K, [N]}\cong \cD_{K}^{-1}\ot_{\OO_{K}}\cG_{K}[\vpi^{N}] =
 {}^{\vee}\cG_{K, [N]}
 \ar@{^{(}->}[r]^{\qquad \qquad j_{N}}& 
 {}^{\vee}\cG_{K}= \cD_{K}^{-1}\ot_{\OO_{K}}\cG_{K}^{},
}
\eeq
  where   the \emph{Tate module}  $T\cG_{K}$ is the $\OO_{K}$-module scheme  
  $$T\cG_{K} :=\varprojlim_{n}\cG_{K}[\vpi^{n}]$$
  with transition maps given by multiplication by $\vpi$. 
   (The isomorphisms denoted by $\cong $ depend on the choice of $\vpi$, which allows to compatibly identify $\OO_{F, [N]}  \cong \OO_{F, [N]}^{*}$. This dependence will be negligible for our purposes .)


\paragraph{Vertical cycles} Let $\lm\in {\bf P}^{1}(F)$.
We attach to $\lm$ the morphisms, all denoted by the same name, 
\beqq f^{\flat}_{\lm}\colon {}^{\prime} \cG_{F, ?} &\to {}^{\prime}\cG_{F, ?}^{2}\\
x &\mapsto x\circ \lm',
\eeqq
where $? \in \{\emptyset, [N]\}$ and $\prime\in \{*,\vee\}$. 

\begin{defi} Let $? \in \{\emptyset,[N]\}$, let  $\prime\in \{*,\vee\}$, and let $\bullet =\emptyset $ if $\prime=\vee$, $\bullet = \circ $ if $\prime=*$. 
The  \emph{vertical cycle} attached to $\lm $ is 
\beqq\
[V^{\flat}(\lm)]^{\bullet}_{?}:= f^{\flat}_{\lm, *} [{}^{\prime}\cG_{F, ?}]^{\bullet }\in  K_{0} ( {}^{\prime}\cG_{F, ?}^{2})_{\Q},
\eeqq
where in the case of $\prime=*$, the normalisation is made viewing $\cG_{F, [N]}$ as a scheme (of degree $d_{N}$) over the base $k$.

  (The symbol $\prime$ is omitted from the name of the cycle, as it is determined by the superscript~$\bullet$.)
\end{defi}

\paragraph{CM cycles}  Let $(\vphi, \tau)\in {\rm LT}(E, F)^{0}$. 
For $n\geq c(\vphi, \tau)$, 
 consider the morphism
\beq 
\vpi^{n} f^{\flat}_{(\vphi, \tau)} \colon {}^{\prime}\cG_{E, ?}  &\to {}^{\prime}\cG_{F,?}^{2} \\
y &\mapsto \vphi\circ y \circ \vpi^{n}\tau'.
\eeq
(Its name is meant to suggest the idea that   $ \vpi^{n} f^{\flat}_{(\vphi, \tau)} $ may be thought of as the composition of $[\vpi^{n}]$ and an, in general non-existent, morphism  $f^{\flat}_{(\vphi, \tau)} $.)

\begin{defi}  
Let $? \in \{\emptyset, [N]\}$, let  $\prime\in \{*,\vee\}$, and let $\bullet =\emptyset $ if $\prime=\vee$, $\bullet = \circ $ if $\prime=*$. 
The \emph{CM cycle} attached to $(\vphi, \tau)$ is
$$[Z^{\flat}(\vphi, \tau)]^{\bullet}_{?}:= { d_{n}^{-1}} \cd (\vpi^{n}f^{\flat}_{(\vphi, \tau)})_{*} [{}^{\prime}\cG_{E,?} ]^{\bullet } \qquad \in \  K_{0}({}^{\prime}\cG_{F, ?}^{2})_{\Q},$$
where in the case of $\prime=*$, the normalisation is made viewing ${}^{*}\cG_{E , [N]}$ as a scheme of degree $ed_{N}$ over the base $k$. 

It is easy to verify that the definition is  independent of $n\geq c(\vphi, \tau)$. 
\end{defi}

For $Z^{\flat}= Z^{\flat}(\vphi, \tau), V^{\flat}(\lm)$, we  lighten the notation by $[Z^{\flat}]^{\bullet }_{N}:= [Z^{\flat}]^{\bullet }_{[N]}.$

\begin{lemm} \lb{comp jN}With reference to the maps in \eqref{jN}, 
for  $Z^{\flat}= Z^{\flat}(\vphi, \tau)$, $ V^{\flat}(\lm)$ we have
$$\pi_{N, * } [Z^{\flat}]^{\circ}= [Z^{\flat}]^{\circ}_{N} = d_{N}^{-1} \cd [Z^{\flat}]_{N}= d_{N}^{-1}\cd j_{N}^{*}[Z^{\flat}].$$
\end{lemm}
\begin{proof}
Clear from the definitions. \end{proof}

\begin{rema}\lb{compare tau Li}
We compare our setup and our CM cycles with those of \cite{li} (specialised to the case $K=E$, $h=1$).
Let ${\rm t}:={\rm tr}_{E/F}\in\Hom_{\OO_{F}}(\OO_{E}, \OO_{F})=:\cD_{E}^{-1}$, and recall from the introduction that $\delta^{-1}$ denotes an $\OO_{E}$-generator of $\cD_{E}^{-1}$. Denote by  $(-)^{\vee_{L}}$ the duality defined in \cite[first paragraph of \S~2 and paragraph after (2.3)]{li}, and denoted by $(-)^{\vee} $ there. 
For $\tau\in  {\rm Isom}_{F}(E, F^{2})$, let 
$$
\tau_{L}:=\tau \delta {\rm t}
\qquad
\in {\rm Isom}_{F}(E, F^{2})\ot_{\OO_{E}} \cD_{E}\ot_{\OO_{E}} \cD_{E}^{-1} = {\rm Isom}_{F}(E, F^{2}).
$$
Then
 $\tau_{L}^{\vee_{L}}={\rm t}^{-1}\tau_{L}^{\vee} =\delta\tau^{\vee}$. Moreover, 
  by comparing \eqref{ht tau} and \cite[(2.13)]{li}, we see that  $(\vphi, \tau)\in {\rm LT}(E, F)^{0}$ if and only if $(\vphi, \tau_{L})$ belongs to set ${\rm Equi}_{1}(E/F)$ of \cite[Definition 2.19]{li}. Finally,
 it follows from the respective definitions that
our cycle $[Z^{\flat}(\vphi, \tau)]$ equals the cycle $Z_{\infty}(\vphi, \tau_{L})$ of \cite[Definition 3.10]{li}.
\end{rema}

\subsection{Approximation} We show that the special cycles in $T\cG_{F}^{2}={}^{*}\cG_{F}^{2}$ approximate those in $\cM_{\infty}$. The comparison is done via the following maps:
\beq
\xymatrix{
{}^{*}\cG_{F}^{2} \ar[r]^{\vpi^{m}a_{F}} & \cM_{\infty} &{}^{*}\cG_{E} \ar[r]^{\vpi^{m}a_{E}} & \cN_{\infty} &{}^{*}\cG_{F} \ar[r]^{\vpi^{m}a_\Ig} & \Ig_{\infty}
 \\
\alpha \ar@{|->}[r] &[\cG_{F}, {\rm id}, \vpi^{m}\alpha] , & \beta \ar@{|->}[r] &[\cG_{E}, {\rm id}, \vpi^{m}\beta], & \gamma \ar@{|->}[r] &[\cG_{F}, {\rm id}, \vpi^{m}\gamma].
}
\eeq
(The notation is meant to evoke  non-existent maps $a_{F}$, $a_{E}$, $a_\Ig$.)
\begin{lemm} \lb{is fo} 
For all $m\geq 1$, the maps $\vpi^{m}a_{F}$, $\vpi^{m}a_{E}$, $\vpi^{m}a_\Ig$  are well-defined 
 closed immersions of finite origin (Definition \ref{fin ind}).  More precisely, for any $N$,  those  maps  respectively originate from  the closed immersions
\beq\lb{trunc af}
\vpi^{m}a_{F}\colon {}^{*} \cG_{F, [N]}^{2}\into \cM_{N+m}, \quad \vpi^{m}a_{E}\colon {}^{*}\cG_{E,[N]}\into \cN_{N+m}, \quad \vpi^{m}a_\Ig\colon  {}^{*}\cG_{F,[N]}\into \Ig_{N+m}.\eeq
\end{lemm}
For $N=0$, the  maps  \eqref{trunc af} are simply the inclusion of the closed point $\Spec k$. 
\begin{proof} The well-definedness  follows from   \cite[Lemma 3.5]{li}. The rest of the statement, at least for $N=0$ (which suffices) and for $a_{F}$ and $a_{E}$, is a special case of \cite[Proposition 3.9 (1)]{li} -- in whose notation (to which we add a subscript `$L$' for clarity) we should take  $h_{L}=2, K_{L}=F_{L}=F$ (for $a_{F}$) or $h_{L}=1, K_{L}=F_{L}=E$ (for $a_{E}$), and $\vphi_{L}=\tau_{L}={\rm id}$.
  The arguments of \emph{loc. cit.} also apply to the (easier) case of $a_\Ig$.
\end{proof}

\begin{prop}  \lb{iscart}
For all $m\geq 1 $ and $n\geq c(\vphi, \tau)$, the following diagrams are Cartesian:
\beqq
\xymatrix{
{}^{*}\cG_{E}\ar[rr]^{\vpi^{n}f^{\flat}_{(\vphi, \tau)}} \ar[d]_{\vpi^{m+n}a_{E}}  &&{}^{*}\cG_{F}^{2} \ar[d]^{\vpi^{m}a_{F}}
& &
{}^{*}\cG_{F}\ar[rr]^{f^{\flat}_{\lm} }\ar[d]_{\vpi^{m}a_\Ig}  &&{}^{*}\cG_{F}^{2} \ar[d]^{\vpi^{m}a_{F}}\\
\cN_{\infty } \ar[rr]^{f_{(\vphi, \tau)}} && \cM_{\infty},
&&
\Ig_{\infty } \ar[rr]^{f_{\lm}} && \cM_{\infty}. 
}\eeqq
\end{prop}
\begin{proof} We verify the statement for the second diagram. The proof for the first diagram is similar and also a special case of \cite[Proposition 3.9 (2)]{li}. It suffices to verify that the diagram is Cartesian on $A$-valued points functorially in objects $A$ of $\cC_{k}$. We write the proof for the functors in infinite level where the idea is clearest.  The $A$-point $\alpha_{0}\colon \OO_{F}^{*,2}\to \cG_{F}$ of ${}^{*}\cG_{F}^{2}$ is a point of the Cartesian product of the diagram (which is a closed subscheme of ${}^{*}\cG_{F}^{2}$ as  $f_{\lm}$ is a closed immersion) if and only if $\vpi^{m}\alpha_{0}  =\gamma\circ\lm$ factors through $\lm$ and, obviously, through $[\vpi^{m}]$. As $\Ker(\lm)$  is divisible, this is equivalent to $\gamma =\vpi^{m}\gamma_{0}$ for some $\gamma_{0}\colon \OO_{F}^{*}\to \cG_{F}$, that is $\alpha_{0}= f_{\lm}(\gamma_{0})$ for $\gamma_{0}\in {}^{*}\cG_{F}$.
\end{proof}
\begin{coro} \lb{compare cycles} 
We have 
\beq
 d_{E,m}\cd(\vpi^{m}a_{F})^{*} [Z(\vphi, \tau)]^{\circ} &= [Z^{\flat}(\vphi, \tau)]^{\circ},\\
d_{\Ig,m} \cd(\vpi^{m}a_F)^{*} [V(\lm)]^{\circ} &= [V^{\flat}(\lm)]^{\circ}.
\eeq
 in $K_{0}({}^{*}\cG_{F}^{2})_{\Q}$.
\end{coro}
\begin{proof}
By Lemma \ref{is fo}, we have  
\beq \lb{comp}
(\vpi^{m}a_{E})^{*}\OO_{\cN_{N+m}}= \OO_{\cG_{E,[N]}}, \qquad (\vpi^{m}a_\Ig)^{*}\OO_{\Ig_{N+m}}= \OO_{\cG_{F,[N]}}
\eeq
 for all $N$. Then, 
\beqq
(\vpi^{m} a_{F})^{*} [Z(\vphi, \tau)]^{\circ}  & =  (\vpi^{m} a_{F})^{*} f_{(\vphi, \tau), *} [\cN_{\infty}]^{\circ}\\
&=  (\vpi^{n}f^{\flat}_{(\vphi, \tau)})_{*} (\vpi^{m+n} a_{E})^{*} [\cN_{\infty}]^{\circ} = d_{E, m+n}^{-1}\cd (\vpi^{n}f^{\flat}_{(\vphi, \tau)})_{*} [{}^{*}\cG_{E}]^{\circ}, \eeqq
where the second equality follows from Proposition \ref{iscart}  
and 
 the pullback-pushforward formula  (Proposition \ref{pull p}),
 and the last one follows from \eqref{comp}. This proves the first identity and the same argument works for the second one.
\end{proof}

\section{Intersection numbers} \lb{sec5}
We wish to compute  the intersection multiplicity 
$$m([Z(\vphi, \tau)_{U}],  [V(\lm)_{U}])= d_{\psi_{\tau}^{-1}(U)} d_{U\cap \OO_{F}^{\ts}}\cd m( [Z(\vphi, \tau)_{U}]^{\circ} , [V(\lm)_{U}]^{\circ}) $$
in $\cM_{U}$ (here, the constants $d_{?}$ are as in \eqref{dck}).
Similarly to what is done in \cite{li} for intersections of CM cycles, we first compare this with an intersection multiplicity in formal groups, then compute the latter.
\subsection{Comparison}\lb{41}
In the next proposition, we compare intersection multiplicities in Lubin--Tate spaces and in formal groups. 

Define for all levels $U$ and integers $N$:
\beq\lb{m0}
m^{\circ}( [Z(\vphi, \tau)_{U}]^{\circ} , [V(\lm)_{U}]^{\circ}) &:= \vol(U)^{-1}\cd m( [Z(\vphi, \tau)_{U}]^{\circ} , [V(\lm)_{U}]^{\circ})\\
m^{\circ}( [Z^{\flat}(\vphi, \tau)]_{N}^{\circ} , [V^{\flat}(\lm)]_{N}^{\circ}) &:= d_{N}^{2} \cd 
m([Z^{\flat}(\vphi, \tau)]_{N}^{\circ} , [V^{\flat}(\lm)]_{N}^{\circ}),
\eeq
where $\vol(U)$ is normalised by $\vol(\GL_{2}(\OO_{F}))=1$.  We note that for all $N\geq 1$, $$\vol(U_{N})= \zeta_{F}(1)\zeta_{F}(2) q^{-4N}.$$
\begin{prop} 
\begin{enumerate}
\item
The limits 
\beq 
m^{\circ}( [Z(\vphi, \tau)]^{\circ} , [V(\lm)]^{\circ}) :=\lim_{U\to \{1\}} m^{\circ}( [Z(\vphi, \tau)_{U}]^{\circ} , [V(\lm)_{U}]^{\circ}) \\
m^{\circ}( [Z^{\flat}(\vphi, \tau)]^{\circ} , [V^{\flat}(\lm)]^{\circ}) := \lim_{N\to \infty} m^{\circ}( [Z^{\flat}(\vphi, \tau)]_{N}^{\circ} , [V^{\flat}(\lm)]_{N}^{\circ}) 
\eeq
exist as limits of eventually constant sequences. 
\item  We have the following identity relating  intersection multiplicities in $\cM_{\infty}$, $(T\cG_{F})^{2}$, and $\cG_{F}^{2}$:
$$m^{\circ}( [Z(\vphi, \tau)]^{\circ} , [V(\lm)]^{\circ})  = {\zeta_{E}(1)\over  \zeta_{F}(2)} \cd m^{\circ}( [Z^{\flat}(\vphi, \tau)]^{\circ} , [V^{\flat}(\lm)]^{\circ}) =  {\zeta_{E}(1)\over e \zeta_{F}(2)} \cd m( [Z^{\flat}(\vphi, \tau)], [V^{\flat}(\lm)]) .$$
\end{enumerate}
\end{prop}
\begin{proof} 
To establish the entirety of the proposition, it is enough to show that for all sufficiently large $N$ and $M$, 
\beq\lb{one23}
{ \zeta_{F}(2)\over \zeta_{E}(1)}\cd m^{\circ}( [Z(\vphi, \tau)]_{N+M}^{\circ} , [V(\lm)]_{N+M}^{\circ})  &=   m^{\circ}( [Z^{\flat}(\vphi, \tau)]_{N}^{\circ} , [V^{\flat}(\lm)]_{N}^{\circ})\\
= m( [Z^{\flat}(\vphi, \tau)]_{N+M}^{} , [V^{\flat}(\lm)]_{N+M}) &=  m( [Z^{\flat}(\vphi, \tau)], [V^{\flat}(\lm)]), \eeq
where for a special  cycle $Z$ in $\cM_{\infty}$ we have denoted  $[Z]^{\circ}_{N}:= [Z^{\circ}_{U_{N}}]$.

The second equality follows trivially from the definitions and Lemma \ref{comp jN}. 
 The third one is proved exactly in the same way as \cite[(4.10)]{li}. We recall Li's idea: the intersection number in $\cG_{F}^{2}$, as the length of a module supported at the closed point of $\cG_{F}^{2}$, is the same as the intersection number of the restrictions to  a sufficiently thick Artinian thickening
$$\Spf \OO_{\cG_{F}^{2}} /\frakm_{\cG_{F}^{2}}^{q^{2N}} \stackrel{\iota_{N}}{\hookrightarrow}  \cG_{F, [N]}^{2}\subset \cG_{F}^{2},$$
 where the natural inclusion $\iota_{N}$ is the map denoted by $j_{M}$ in \cite[(4.16)]{li} (with the $M$ of 
 \emph{loc. cit.} being equal to our $N$). Here, ``sufficiently thick'' means that $q^{2N}$ should be greater than the intersection number.

We now  prove the first equality based on the argument to prove \cite[(4.8), (4.9)]{li}, only writing the details that are different from \emph{loc. cit.} By (the proof of) Corollary \ref{compare cycles}, for all sufficiently large $N$ and  $M$ we have 
$${d_{E,M}d_{\Ig, M} } d_{2N} \cd m\left((\vpi^{M}a_{F})^{*} [Z(\vphi, \tau)]^{\circ}_{N+M}, (\vpi^{M}a_{F})^{*} [V(\lm)]^{\circ}_{N+M}\right) = d_{2N}\cd m \left(  [Z^{\flat}(\vphi, \tau)]^{\circ}_{N}, [V^{\flat}(\lm)]^{\circ}_{N}\right)$$
as intersections in $(\cG_{F,[N]})^{2}$.
Thus it suffices to observe that
  $${d_{E,M}d_{\Ig, M} } d_{2N} =  \zeta_{E}(1)^{-1}\zeta_{F}(1)^{-1}\cd q^{4(N+M)} =  \zeta_{E}(1)^{-1}  \zeta_{F}(2) \cd \vol(U_{N+M})^{-1},$$  and to show that 
\beq \lb{art}m((\vpi^{M}a_{F})^{*} [Z(\vphi, \tau)]^{\circ}_{N+M}, (\vpi^{M}a_{F})^{*} [V(\lm)]^{\circ}_{N+M}) =
m( [Z(\vphi, \tau)]_{N+M}^{\circ} , [V(\lm)]_{N+M}^{\circ}) .\eeq

The argument on  restrictions to sufficiently thick Artinian thickenings  sketched above for the first equality in \eqref{one23} also applies to prove \eqref{art}: by \cite[\S 4.4]{li}, $\vpi^{M}a_{F}\circ \iota_{N}$ identifies $\Spf \OO_{\cM_{N+M}}/\frakm_{{\cM_{N+M}}}^{q^{2N}}\subset \cM_{N+M}$ with $\Spf \OO_{\cG_{F}^{2}} /\frakm_{\cG_{F}^{2}}^{q^{2N}}\subset \cG_{F}^{2}$.
\end{proof}

\begin{coro}\lb{fininf} We have 
$$m^{\circ}( [Z(\vphi, \tau)_{U}]^{\circ} , [V(\lm)_{U}]^{\circ}) =
 {\zeta_{E}(1)\over  \zeta_{F}(2)} \cd\vol(U)^{-1}
\cd
\int_{U}
 m( [Z^{\flat}(\vphi, \tau)], [V^{\flat}(\lm g)])\,  dg.$$
\end{coro}
\begin{proof} By the proposition, we may replace the integrand with $ m^{\circ}( [Z(\vphi, \tau)_{U'}]^{\circ}, [V(\lm g)_{U'}]^{\circ})$ for sufficiently small $U'$. Then the result follows from the projection formula and the observation that $$\pi_{U'U}^{*}[V(\lm)_{U}]^{\circ}= \pi_{U'U}^{*}\pi_{U'U, *}[V(\lm)_{U'}]^{\circ}=\sum_{g\in  U/U'} [V(\lm g)_{U'}]^{\circ}.$$
\end{proof}

\subsection{Computation}
We first  write CM (respectively vertical) cycles in $\cG_{F}^{2}$ as rational  multiples  of cycles that are  image (respectively kernel) of a map  in  $\Hom_{\OO_{F}} (\cG_{F}, \cG_{F}^{2}) =\OO_{D}^{2}$ viewed as column vectors (respectively $\Hom_{\OO_{F}} (\cG_{F}^{2}, \cG_{F})= \OO_{D}^{2}$ viewed as row vectors). 

\paragraph{Notation}
 By the identification $\cG_{E}=\cG_{F}$ as $\OO_{F}$-modules, we have an embedding 
 \beq \lb{embdd}
 \OO_{E}=\End_{\OO_{E}\text{-Mod}} (\cG_{E})\into \OO_{D}=\End_{\OO_{F}\text{-Mod}}(\cG_{F}).\eeq
Therefore both $M_{2}(E)$ and $D$ are embedded in 
\beqq
M_{2}(D)=\Hom_{\OO_{F}} (\cG_{F}^{2}, \cG_{F}^{2})\ot F .\eeqq
We denote by $\nrd$ the reduced norm of $D$, which restricts to the relative norm $N_{E/F}$  of $E/F$ under \eqref{embdd}.

Recall from the introduction that we denote by   $\lm^{\perp}\colon \OO_{F}^{2}\onto \OO_{F}$ a map such that $\Ker(\lm^{\perp})= {\rm Im}(\lm)$, and by   $\delta^{-1}\in \cD_{E}^{-1}=\Hom_{\OO_{F}}(\OO_{E}, \OO_{F})$ a generator. For $\tau\in {\rm Isom}_{F}(E, F^{2})$, write 
\beq \lb{lm tht}
\tau\delta={\tht_{1} \choose  \tht_{2}}\in \Hom (\OO_{E}, \OO_{F}^{2}) \ot_{\OO_{E}}\cD_{E} \ot_{\OO_{F}}F= 
E^{2}, \qquad \lm^{\perp} = (a , b)\in (\OO_{F}^{2})^{\vee}.
\eeq
Let
\beq\lb{iw dec}
( \tau\delta \  | \ \baar{ {\tau}\delta}) = \Gamma_{\tau} \twomat {P_{\tau}} *{} {Q_{\tau}}, \qquad  \Gamma_{\tau}\in \GL_{2}(\OO_{E}), \eeq
be an  Iwasawa decomposition in $\GL_{2}(E)$.

\begin{lemm} With notation as in \eqref{lm tht}, \eqref{iw dec}, we have 
\beqq
\ [V^{\flat}(\lm)] &=[{\rm \Ker}( (a\ b))],\\ 
 [Z^{\flat}(\vphi, \tau)]&= |\nrd (\vphi P_{\tau})|^{-1}\cd [{\rm Im}(\vphi \Gamma_{\tau}\vphi^{-1} {1\choose 0})],
\eeqq
where the map  $\vphi \Gamma_{\tau}\vphi^{-1} {1\choose 0}\colon \cG_{E} \to\cG_{F}^{2} $  is a closed immersion.
\end{lemm}
\begin{proof} The first  equality is clear.
The second one is a special case of \cite[(5.3)]{li}; under the comparison of setups from Remark \ref{compare tau Li}, our vector ${\theta_{1}\choose \theta_{2}}$ corresponds to the vector $M_{\tau_{L}}$ of \cite[Definition 5.4]{li}. 
\end{proof}

We may now compute the intersection in $\cG_{F}^{2}$. 
\begin{prop}  With notation as in \eqref{lm tht}, \eqref{iw dec}, we have 
$$m([Z^{\flat}(\vphi, \tau)],  [V^{\flat}(\lm)] ) = |\nrd(\vphi)|^{-1} \cd   | \lm^{\perp}   \tau\delta |_{E}^{-1}.$$
\end{prop}
\begin{proof}
By the previous lemma, $m([Z^{\flat}(\vphi, \tau)] ,[V^{\flat}(\lm)] )$ equals   $|\nrd(\vphi P_{\tau})|^{-1}$ times the
 degree of the group scheme
 kernel of 
\beq \lb{the iso}
(a\ b) \circ \vphi \Gamma_{\tau} \vphi^{-1} {1 \choose  0}\colon \cG_{F}\to \cG_{F},\eeq
that is, the degree of the isogeny \eqref{the iso}. By \cite[Lemma 5.7]{li}, this is
$$| \nrd ((a\ b) \vphi \Gamma_{\tau} \vphi^{-1} {1 \choose  0})|^{-1}= | \nrd ((a\ b)  \Gamma_{\tau}  {1 \choose  0})|^{-1}.$$
Thanks to  the decompositions
\beqq
(\tau\delta\  \ \baar{\tau\delta} ) :=  \twomat {\tht_{1}} {\baar{\tht}_{1}} {\tht_{2}} {\baar{\tht}_{2}}
&= \twomat 1  {} {\tht_{2}/\tht_{1}}  1  \twomat {\tht_{1}} {\baar{\tht}_{1}} {}  {\baar{\tht}_{2}- \baar{\tht}_{1} \tht_{2}/ \tht_{1}}
\\
  &= 
 \twomat  {\tht_{1}/\tht_{2}}  1  1{}
  \twomat {\tht_{2}} {\baar{\tht}_{2}}
 {}  {\baar{\tht}_{1}- {\tht}_{1} \baar{\tht}_{2}/ \tht_{2}}
\eeqq
we may compute the elements $\Gamma_{\tau}$, $P_{\tau}$ from  \eqref{iw dec}. We only write down the details if $v(\tht_{1}) \leq v(\tht_{2})$: then the first decomposition is an Iwasawa decomposition,  and 
\beqq
m([Z^{\flat}(\vphi, \tau)], [V^{\flat}(\lm)] ) &=|\nrd(\vphi \tht_{1})|^{-1}  |\nrd((a \ b) \smalltwomat 1  {} {\tht_{2}/\tht_{1}}  1 \textstyle{1\choose 0})|^{-1} \\ 
&= |\nrd(\vphi)|^{-1} \cd |N_{E/F}(a\tht_{1}+b\tht_{2})|^{-1},
\eeqq
which is the  desired formula.
\end{proof}

Inserting this result into Corollary \ref{fininf} and removing the constants indicated by  the superscripts~`$\circ$', we obtain the following local analogue of Theorem \ref{thm B}.

\begin{coro} \lb{int fo} 
We have 
$$m( [Z(\vphi, \tau)_{U}]^{} , [V(\lm)_{U}]^{}) =
 {\zeta_{E}(1)\over  \zeta_{F}(2)}\cd 
 d_{\psi_{\tau}^{-1}(U)}  d_{U\cap\OO_{F}^{\ts}}
\cd |\nrd(\vphi)|^{-1}\cd
 \int_{U}  | \lm^{\perp}  g   \tau\delta  |_{E}^{-1}\, dg,$$
  where $dg$ is the  Haar measure such that $\vol(\GL_{2}(\OO_{F}))=1$.
\end{coro}

It is a not unpleasant exercise to verify  that, for $U=\GL_{2}(\OO_{F})$, this formula agrees with 
 \eqref{int qcan}.

\subsection{Conclusion} We are now in a position to complete the proof of the main theorems.

\begin{coro}\lb{cor bounded} 
Let $V\neq V'$ be irreducible components of $\cM_{U, k}$. The intersection multiplicity 
$m(Z, V')$
is bounded as  $Z$ varies among  CM curves with generic fibre a point  of  $\cB_{V}$.
\end{coro}
A qualitatively similar result from a somewhat different point of view can be deduced from \cite[Proposition 4.3.3]{nonsplit};  the relation `$z\in B_{V}$' should be compared to the relation `$V(z)=V$' in the notation of \emph{loc. cit.}
\begin{proof} 
 Let $n$ be such that $U\supset U_{n}$, let $V=V(\lm)_{U}$, $V'=V(\lm')_{U}$, and let $z=z(\vphi, \tau)_{U} \in B_{V}$ have conductor $s\geq n$.   Up to changing the choice of representatives  $\vphi$ and $\tau$,
  we may assume that
$\OO_{F}^{2}\subset \tau(\OO_{E})\subset  \vpi^{-s} \OO_{F}^{2}$,
and ${\rm ht}(\vphi) = s$. 

Suppose  that $z=z(\vphi , \tau)\in \cB_{V(\lm)_{U}}$ is a CM point of conductor $s$ and write $\tau\delta={\tht_{1} \choose \tht_{2}}$ as in \eqref{lm tht}.
 Up to changing the choice of representatives  $\vphi$ and $\tau$,
  we may assume that
\beqq
\OO_{F}^{2}\subset \tau(\OO_{E})\subset  \vpi^{-s} \OO_{F}^{2}
\eeqq
 optimally (that is,  $\vpi^{-1}\OO_{F}^{2} \not\subset \tau(\OO_{E}) \not\subset  \vpi^{1-s} \OO_{F}^{2}$),
and ${\rm ht}(\vphi) = s$. We examine the terms of the multiplicity formula of  Corollary \ref{int fo}. 

We have $|\nrd(\vphi)|^{-1}=q^{s}$ and
 $$ d_{\psi_{\tau}^{-1}(U)} \leq d_{\psi_{\tau}^{-1}(U_{0})}= [\OO_{E}^{\ts} :\psi_{\tau}^{-1}(U_{0})] =e\zeta_{E}(1)\zeta_{F}(1)^{-1} \cd q^{s}.$$ 
Thus it suffices to show that the integrand $| \lm'^{\perp} g  \tau \delta  |_{E}^{-1}$ is bounded by a constant  multiple of $q^{-2s}=|\vpi^{-s}|^{-1}_{E}$.
 Let $v$ be the valuation on $E$ normalised by $v(\vpi)=1$ (thus $v$ may take half-integer values if $E/F$ is ramified). We will prove more precisely that
\beq v(\lm'^{\perp}g \tau\delta)
 \geq n-s \quad \Longrightarrow \quad z(\vphi , \tau)\in \cB_{V(\lm')_{U}}.\eeq
 
 In fact  $v(\lm'^{\perp} g  \tau\delta) \geq n-s $
 if and only if $\lm'^{\perp}g$ takes integral values on $ {\rm Im}(\vpi^{s-n}\tau)\subset \vpi^{-n}\OO_{F}^{2}/\OO_{F}^{2}$.  Equivalently, the first among  the lines 
$$ L_{n}(\tau) := {\rm Im}(\vpi^{s-n}\tau)/\OO_{F}^{2}, \quad 
L_{n}(g\lm'):= \Ker [\vpi^{-n}\OO_{F}/\OO_{F})^{2}\stackrel{\lm'^{\perp}g}{\rightarrow} \vpi^{-n}\OO_{F}/\OO_{F}] \quad \subset (\vpi^{-n}\OO_{F}/\OO_{F})^{2} $$
is contained in the second one, hence the two coincide. Up to changing $g$ by a $g'\in U_{n}\subset U$, the analogous lines $L_{s}(\tau)$ and $L_{s}(g\lm')$ also coincide; by definition (see \eqref{Ln}),  this  means that $z(\vphi , \tau)\in \cB_{V(\lm')_{U}}$. \end{proof}

\begin{proof}[Proof of Theorems \ref{main thm} and \ref{thm B}] 
By Lemma \ref{compare basins 2}, Corollary \ref{int fo} implies  Theorem \ref{thm B} (whose constant $c_{E_{v}}= {\zeta_{E_{v}}(1)/  \zeta_{F_{v}}(2)}$),
 and Corollary \ref{cor bounded}  implies Proposition \ref{bounded}. As shown in \S~\ref{sec 22}, the latter implies Theorem \ref{main thm}.
\end{proof}

\appendix
\section{Cycles in formal schemes and their projective limits}\lb{sec3}
\subsection{Pro-Noetherian formal schemes}
 Consider the category $\mathscr{F}$  of separated Noetherian  finite-dimensional formal  schemes   and finite flat maps. We let $\widehat{\mathscr{F}}$ be the pro-category  of $\mathscr{F}$; that is, objects of $\widehat{\cF}$ are formal inverse  limits
$$X=\varprojlim_{i\in I} X_{i}$$ 
of  filtered inverse systems of objects and maps in $\cF$, and morphisms are defined by 
\beq \lb{lim hom}
\Hom(X, Y) = \varprojlim_{j} \varinjlim_{i} \Hom (X_{i}, Y_{j}).\eeq
Henceforth we will simply refer to objects of $\cF$ (respectively $\widehat{\cF}$) as Noetherian formal schemes (respectively pro-Noetherian formal schemes) for short.

\begin{defi} \lb{fin ind} We say that a morphism $f\colon X\to Y=\varprojlim_{j\in J}Y_{j}$  in $\widehat{\cF}$ is 
\begin{itemize}
\item
\emph{of finite origin} if  there exist
a map $f_{0}\colon X_{0} \to Y_{0}$ in $\cF$, a map $Y\to Y_{0}$ in $\widehat{\cF}$, and an isomorphism $$X\cong \varprojlim_{j\in J} X_{{0}}\ts_{Y_{0}} Y_{j}$$
over $Y$. 
In this  case we say that $f$ \emph{originates from} $f_{{0}}$.
\item
\emph{a closed immersion} of finite origin if it is of finite origin and it originates from a closed immersion in $\cF$. 
\end{itemize}
\end{defi}
The \emph{image} of a morphism $\lim f_{ij}\colon X\to Y$ as in \eqref{lim hom} is the object $\varprojlim f_{j}(X)$ with respect to the restrictions of the transition maps of $(Y_{j})$, where $f_{j}(X):= f_{ij}(X_{i})$ for any sufficiently large~$i$.

If $f\colon X\to Z $, $g\colon Y\to Z $ are morphisms in $\widehat{\cF}$ with $g$ of finite origin, the Cartesian product $X\ts_{Z} Y $ is (well-)defined, up to isomorphism, as follows.  Assume that $g$ originates from $g_{0}\colon Y_{0}\to Z_{0}$, and write $X=\varprojlim X_{i}$.
 Then
$$X\ts_{Z}Y:= \varprojlim_{i} X_{i}\ts_{Z_{0}} Y_{{0}}.$$

\subsection{Cycles,  K-groups, intersections}
Let $X$ be a   Noetherian  formal scheme. We define the $K$-group of $X$  with $\Q$-coefficients
$$K_{0}(X)_{\Q}$$
in the usual way as the Grothendieck group (tensored with $\Q$) of the category of coherent sheaves on $X$ (cf. \cite[\S 15.1]{ful}).  If $X=\varprojlim_{i} X_{i}$ is a pro-Noetherian formal scheme, we define 
$$K_{0}(X)_{\Q} := \varprojlim_{i} K_{0}(X_{i})_{\Q}$$
where the limits are with respect to pushforward.

\paragraph{Fundamental classes} Let $X$ be a Noetherian formal scheme. Its  fundamental class is 
$$[X]:= [\OO_{X}] \in K_{0}(X)_{\Q}.$$

Let $X$ be a pro-Noetherian formal scheme  with a  morphism to a Noetherian formal scheme  $B$; equivalently, for  all sufficiently large $i$ we have a finite flat map $X_{i}\to B$
 compatibly in the system. We define the normalised (relative to $B$)  fundamental class of $X$
to be 
$$[X]^{\circ}:=\lim [X_{i}]^{\circ}\in K_{0}(X)_{\Q}, \qquad [X_{i}]^{\circ}:= {1\over \deg (X_{i}/B)} [X_{i}] \qquad \in K_{0} (X_{i})_{\Q}.$$

\subsubsection{Pullback and pushfoward} It is easy to verify that  $K_{0}(-)_{\Q}$ enjoys  pushforward  (respectively, pullback) functoriality with respect to arbitrary (respectively, finite-origin) morphisms in $\widehat{\cF}$.

\begin{prop} \lb{pull p}
Let $f\colon X\to Z$, $g\colon Y\to Z$ be morphisms in $\widehat{\cF}$ with $g$ of finite origin, and consider the Cartesian diagram 
\begin{equation*}
\xymatrix{
 X\ts_{Z}Y\ar[r]^{f'}\ar[d]^{g'} & Y\ar[d]^{g}\\
X\ar[r]^{f} & Z.
}
\end{equation*}
 
 For $C\in  K_{0}(X)_{\Q}$, we have
 $$ g^{*}f_{*}C = f'_{*}g'^{*} C.$$
\end{prop}
\begin{proof} Write $X=\varprojlim_{i}X_{i}$, $C=\varprojlim C_{i}$,  and assume that $g$ originates from $g_{0}\colon Y_{0}\to Z_{0}$. For sufficiently large $i$, let  $f_{i}\colon X_{i}\to Z_{0}$ be the relevant component of $X\stackrel{f}{\to}Z\to Z_{0}$, and consider the Cartesian diagram
\begin{equation*}
\xymatrix{
 X_{i}\ts_{Z_{0}}Y_{0}\ar[r]^{f_{i}'}\ar[d]^{g_{0}'} & Y_{0}\ar[d]^{g_{0}}\\
X_{i}\ar[r]^{f_{i}} & Z_{0}.
}
\end{equation*}
By the definitions, the desired formula reduces to $g_{0}^{*}f_{i, *} C_{i} = f_{i, *}' g_{0}'^{*} C_{i}$, which is just the  usual pullback-pushforward formula in $\cF$.
\end{proof}

\subsubsection{Intersection multiplicities} 
If $X$ is a Noetherian formal scheme and $x\in X$ is a regular point,  the  intersection multiplicity function
$$m_{x}\colon K_{0}(X)_{\Q}\ts K_{0}(X)_{\Q}\to \Q$$
is defined via Serre's Tor formula \cite[\S 20.4]{ful}. It is a symmetric bilinear form such that if $X=\X'$ and $Z_{1}$, $Z_{2}$ are as in \eqref{int prop}, then 
\beqq
m_{x}([Z_{1}], [Z_{2}])=m_{x}(Z_{1}, Z_{2})
\eeqq
where the right-hand side is the plain intersection multiplicity defined in \eqref{int prop}.
The subscript $x$ is omitted when $|X|$ consists of a a single point.

\backmatter
\addtocontents{toc}{\medskip}

\end{document}